\documentclass[a4paper,12pt,reqno,twoside]{amsart}

\usepackage[utf8]{inputenc} 		
\usepackage[T1]{fontenc} 

\usepackage{amssymb}
\usepackage{amsmath}
\usepackage{graphicx}
\usepackage[colorlinks,citecolor=blue,urlcolor=blue]{hyperref}
\usepackage{cite}
\usepackage[hmargin=2.5cm,vmargin=2.5cm]{geometry}
\usepackage{mathrsfs} 
\usepackage{ellipsis}

\newcommand{\be}{\begin{eqnarray}}
\newcommand{\ee}{\end{eqnarray}}

\newcommand{\beq}{\begin{equation}}
\newcommand{\eeq}{\end{equation}}
\newcommand{\beqn}{\begin{equation*}}
\newcommand{\eeqn}{\end{equation*}}

\newcommand{\round}[1]{\lfloor#1\rfloor}

\DeclareMathOperator{\tr}{tr}

\newtheorem{thm}{Theorem}[section]
\newtheorem{prop}[thm]{Proposition}

\newtheorem{lem}[thm]{Lemma}
\newtheorem{defn}[thm]{Definition}

\newtheorem{remark}[thm]{Remark}

\newcommand\cA{{\mathcal A}}
\newcommand\cB{{\mathcal B}}
\newcommand\cC{{\mathcal C}}

\newcommand\cE{{\mathcal E}}
\newcommand\cF{{\mathcal F}}
\newcommand\cG{{\mathcal G}}
\newcommand\cH{{\mathcal H}}
\newcommand\cI{{\mathcal I}}

\newcommand\cL{{\mathcal L}}

\newcommand\cN{{\mathcal N}}

\newcommand\cT{{\mathcal T}}

\newcommand\bE{{\mathbb E}}

\newcommand\bN{{\mathbb N}}

\newcommand\bP{{\mathbb P}}

\newcommand\bR{{\mathbb R}}

\newcommand\bZ{{\mathbb Z}}

\newcommand\rd{{\mathrm d}}

\newcommand{\ve}{\varepsilon}

\def\bfT{\mathbf{T}}

\begin{document}

\title[Sunklodas' approach to normal approximation for dynamical systems]{Sunklodas' approach to normal approximation for time-dependent dynamical systems}


\author[Juho Lepp\"anen]{Juho Lepp\"anen}
\address[Juho Lepp\"anen]{
LPSM, Laboratoire de Probabilit\'es, Statistique et Mod\'elisation, Sorbonne Universit\'e, 4 Place Jussieu, 75252 Paris, France}
\email{leppanen@lpsm.paris}

\author[Mikko Stenlund]{Mikko Stenlund}
\address[Mikko Stenlund]{
Department of Mathematics and Statistics, P.O.\ Box 68, Fin-00014 University of Helsinki, Finland.}
\email{mikko.stenlund@helsinki.fi}
\urladdr{http://www.helsinki.fi/~stenlund/}


 





\begin{abstract} We consider time-dependent dynamical systems arising as sequential compositions of self-maps of a probability space. We establish conditions under which the Birkhoff sums for multivariate observations, given a centering and a general normalizing sequence $b(N)$ of invertible square matrices, are approximated by a normal distribution with respect to a metric of regular test functions. Depending on the metric and the normalizing sequence $b(N)$, the conditions imply that the error in the approximation decays either at the rate $O(N^{-1/2})$ or the rate $O(N^{-1/2} \log N)$, under the additional assumption that $\Vert b(N)^{-1} \Vert \lesssim N^{-1/2}$. The error comes with a multiplicative constant whose exact value can be computed directly from the conditions. The proof is based on an observation due to Sunklodas regarding Stein's method of normal approximation. We give applications to one-dimensional random piecewise expanding maps and to sequential, random, and quasistatic intermittent systems.
\end{abstract}

\maketitle


\subsection*{Acknowledgements} JL was supported by DOMAST (University of Helsinki) and by the European Research Council (ERC) under the European Union's Horizon 2020 research and innovation programme (grant agreement No 787304). He would like to thank Viviane Baladi for helpful discussions. MS was supported by Emil Aaltosen S\"a\"ati\"o, and the Jane and Aatos Erkko Foundation


\section{Introduction}\label{sec:intro}

In this note we revisit the topic of statistical limit laws by Stein's method for dynamical systems, studied previously in \cite{HLS_2016, Hella_2018, HL_2018, King_1997, Gordin_1990, Denker_etal_2004, Psiloyenis_2008, GordinDenker_2012, Haydn_2013, HaydnYang_2016}. We consider discrete time-dependent dynamical systems described by sequential compositions $\cT_n = T_n \circ \cdots \circ T_1$, where each $T_i : X \to X$ is a transformation of a probability space $(X,\cB,\mu)$. The measure $\mu$ is not assumed to be invariant under any of the maps $T_i$. Given a bounded obsevable $f : X \to \bR^d$ and a sequence $b = b(N) \in \bR^{d \times d}$ of invertible matrices, we are interested in approximating the law of the sums
\begin{align*}
W = W(N) = \sum_{n=1}^{N-1} b^{-1} ( f \circ \cT_n - \mu(f \circ \cT_n) )
\end{align*}
by a multivariate normal distribution. More precisely, we want to identify conditions that cover a wide range of chaotic time-dependent systems and imply a good upper bound on
\begin{align}\label{eq:intro_aim}
\sup_{h \in \cH} | \mu[ h(W) - \Phi_{\Sigma}(h)] |,
\end{align}
where $\cH$ is a class of regular test functions $h : \bR^d \to \bR$, and $\Phi_{\Sigma}(h)$ denotes the expectation of $h$ with respect to the multivariate normal distribution $\cN(0, \Sigma)$ with covariance matrix $\Sigma = \text{Cov}_\mu(W) = \mu(W \otimes W)$.

Since its introduction in \cite{Stein_1972}, Stein's method has seen extensive development in the literature of probability theory. In the present context of dynamical systems, the simple basic idea of the method can be described as follows. If for each test function $h \in \cH$ the solution $A : \bR^d \to \bR$ to the differential equation (called a Stein equation)
\begin{align}\label{eq:stein_multi}
\tr  \Sigma D^2A(w) - w^T  \nabla A(w) = h(w) - \Phi_{\Sigma}(h) \hspace{0.5cm} (w \in \bR^d)
\end{align}
lies in another class of functions $\cA$, then it follows that
\begin{align}\label{eq:stein_ineq}
\eqref{eq:intro_aim} \le \sup_{A \in \cA} |\mu[ \tr  \Sigma D^2A(W) - W^T  \nabla A(W) ]|.
\end{align}
In this way the original problem of approximating the law of $W$ by $\cN(0, \Sigma)$ is reduced to bounding the right hand side of  \eqref{eq:stein_ineq}, which interestingly only depends on the law of $W$ and the class $\cA$. It  was observed in \cite{HLS_2016,Hella_2018} that, when $b(N) = \sqrt{N} I_{d\times d}$, Taylor expanding $\nabla A(W)$ about the punctured sums 
\begin{align*}
W^{n,K} = \sum_{0 \le i \le N -1 \,:\, | i - n| > K} b^{-1} ( f \circ \cT_i - \mu(f \circ \cT_i) )
\end{align*}
with a suitably chosen $K = K(N) \gg 1$ leads to certain correlation-decay conditions for an upper bound on $|\mu[ \tr  \Sigma D^2A(W) - W^T\nabla A(W) ]|$. Such an approach calls for bounds on partial derivatives of $A$, which are known to follow from bounds on partial derivatives of $h$. In \cite{HLS_2016,Hella_2018}, $\cH$ was taken to be the class of three times differentiable functions with bounded derivatives in the case of a general $d > 1$, and the class of Lipschitz continuous functions in the case $d=1$. 

The approach described above was applied in \cite{HLS_2016}  to stationary Sinai billiards and in \cite{Hella_2018} to time-dependent smooth uniformly expanding circle maps. Both systems are (the latter in a certain sense) exponentially mixing, which is  essentially the reason why replacing $W$ with $W^{n,K}$ in the application of Stein's method causes only a small error. Indeed, upper bounds of order $O( N^{-1/2} \log N)$ on \eqref{eq:intro_aim} for sufficiently regular observables could be obtained this way. While such a ``fixed gap'' approach works also for polynomially mixing systems, it yields a larger error depending on the rate of mixing. This can be seen from the results of \cite{HL_2018}, where time-dependent systems in the spirit of \cite{aimino2015,Nicol_2018} described by sequential compositions $T_{\alpha_n} \circ \cdots \circ T_{\alpha_1}$ of polynomially mixing intermittent maps $T_{\alpha_n} : [0,1] \to [0,1]$ with parameters $0 \le \alpha_n \le \beta_* < 1/3$ were considered. Under the condition that $\Sigma = \text{Cov}_\mu(W)$ is  positive definite,  an upper bound of order $O(N^{\beta_* - 1/2} (\log N)^{1/\beta_*})$ was obtained for Lipschitz continuous observables. The result was used to establish central limit theorems for quasistatic and random compositions of intermittent maps. 

The purpose of the present note is to describe an adaptation of Stein's method that is more suitable than those of \cite{HLS_2016,Hella_2018} for normal approximation of polynomially mixing systems, and investigate some of its implications. The starting point is a decomposition of $\mu[ \tr  \Sigma D^2A(W) - W^T\nabla A(W) ]$ due to Sunklodas \cite{sunklodas2007}, which allows to identify correlation-decay conditions that imply a rate of decay for 
$\eqref{eq:intro_aim}$ depending on the ``growth of $b(N)$''. In the case of a general $b(N)$ such that $\Vert b(N)^{-1} \Vert \lesssim N^{-1/2}$, the conditions yield the rate $O(N^{-1/2})$ for a class of smooth test functions $\cH$, and in the special self-norming case $b(N) = [ \text{Cov}_\mu( \sum_{n < N} ( f \circ \cT_n - \mu(f \circ \cT_n) ) ) ]^{1/2}$  the rate $O(N^{-1/2} \log N)$ for Lipschitz continuous test functions. A key ingredient in the proof of the latter estimate is a recent result due to Gallou\"et--Mijoule--Swan \cite{Gallouet_2018} concerning the regularity of solutions to Stein equation. As applications we establish rates of convergence in the central limit theorem for the random piecewise expanding model studied  by Dragi\v{c}evi\'{c} et al. in \cite{Froyland_2018} and for sequential, random, and quasistatic intermittent systems. The results for intermittent systems notably improve those of \cite{HL_2018}.

Statistical properties of time-dependent dynamical systems have been studied in several previous works including \cite{Su_almost_2019,rousseau2016, Freitas_2018, kawan2015,tanzi2016,Kawan_2016, SYZ_2013}. Central limit theorems were shown by Bakhtin \cite{bakhtin1994,bakhtin_1994_2}, Conze--Raugi \cite{conze2007}, and more recently by N\'andori--Sz\'asz--Varj\'u \cite{nandori2012} and Nicol-- T\"{o}r\"{o}k--Vaienti 
\cite{Nicol_2018}. Heinrich \cite{Heinrich_1996} showed a Berry-Esseen type upper bound for sequences of uniformly expanding interval maps admitting a Markov partition. Haydn--Nicol--T\"or\"ok--Vaienti \cite{haydn2017} established almost sure invariance principles (ASIP) for piecewise-expanding and other related models, also in higher dimension. ASIPs were obtained also by Castro--Rodrigues--Varandas \cite{Castro_2017} for convergent sequences of Anosov diffeomorphisms and expanding maps on compact Riemannian manifolds. Recently Su \cite{Su_vector_2019} proved a vector valued ASIP for a general class of polynomially mixing time-dependent systems. Among its many implications is a self-norming CLT for the sequential intermittent system with $\beta_* < 1/2$, under a (polynomial) variance growth condition. Finally, Hafouta \cite{Hafouta_2019} showed several limit theorems, including a Berry-Esseen theorem and a local limit theorem, for sequential compositions of maps belonging to a certain class of distance expanding maps of a compact metric space.

\subsection*{Notation.} For a function \(A: \, \bR^d \to \bR\), we write \(D^kA\) for the $k$th derivative of \(A\), and also denote \(\nabla A = D^1 A\). We define
\begin{align*}
\Vert D^k A \Vert_{\infty} = \max \{ \|\partial_1^{t_1}\cdots\partial_d^{t_d}A_\alpha\|_{\infty}  :  t_1+\cdots+t_d = k,\, 1\le \alpha\le d'  \}.
\end{align*} 

The spectral norm of a matrix $A \in \bR^{d \times d}$ is denoted by 
\begin{align*}
\Vert A \Vert_s = \sup \{ \Vert Ax \Vert : \Vert x \Vert = 1   \},
\end{align*}
where $\Vert \cdot \Vert$ is the Euclidean norm of $\bR^d$. We use $B_d(x,r)$ to denote the open ball in $\bR^d$ with center $x$ and radius $r > 0$.

Given a measure space $(X,\cB,\mu)$ and a $\mu$-integrable function $f : X \to \bR^d$ we set $\mu(f) = \int f \, d \mu$. The components of $f$ are denoted by $f_{\alpha}$, where $\alpha \in \{1,\ldots,d \}$. The Lebesgue measure is denoted by $m$.

For two vectors $v,w \in \bR^d$ we set $v \otimes w = [v_{\alpha}w_{\beta} ]_{\alpha,\beta}$.

We denote by $C$ a generic positive constant whose value might change from one line to the next. We use $C(a_1,\ldots,a_n)$ to denote a  positive constant that depends only on the parameters $a_1,\ldots,a_n$.

\subsection*{Structure of the paper} In Section \ref{sec:main} we present our main results concerning normal approximation of abstract discrete time-dependent dynamical systems. Sections \ref{sec:pw} and \ref{sec:int} contain applications to one-dimensional dynamics. The model of Section \ref{sec:pw} is a random dynamical system of piecewise smooth uniformly expanding maps, while in Section \ref{sec:int} we consider sequential, quasistatic, and random intermittent systems. Finally, in Section \ref{sec:proof} we prove the main results.

\section{Main results}\label{sec:main}

Consider a sequence $(T_n)_{n \ge 1}$ of measurable maps $T_n : X \to X$ of a probability space $(X,\cB,\mu)$. For each $i \ge 0$ let $g^i:X\to\bR^d$ be a bounded measurable function and define
\beqn
f^i = g^i\circ T_i\circ\cdots\circ T_1 \quad\text{and}\quad \bar f^i = f^i - \mu(f^i).
\eeqn

Given $N \ge 1$ and an invertible matrix \(b = b(N) \in \bR^{d \times d}\), we write 
\beqn
W = W(N) = \sum_{i=0}^{N-1} b^{-1} \bar{f}^i.
\eeqn
Given also $n, k \ge 0$, we write
\beqn
\bar{f}^{n,k} = \sum_{0\le i<N : |i - n| = k} \bar{f}^i.
\eeqn

The covariance matrix of $W$ is denoted by 
\beqn
\Sigma = \text{Cov}_\mu(W) = \mu(W \otimes W). 
\eeqn

\subsection{General normalization} First we consider a general invertible $b = b(N)$ and give conditions that imply an upper bound on the distance between the law of $W$ and the normal distribution $\cN(0,\Sigma)$ with respect to a smooth metric. 

Suppose that $\Vert g^i \Vert_{\infty} = \sup_{x \in X} \Vert g^i(x) \Vert \le M$ for all $0 \le i \le N-1$. Then, given a smooth test  function $h : \bR^d \to \bR$ and $(s,t,z) \in [0,1]^2 \times \bR^d$ we define the matrix-valued function $G_{h} = G_h^{(s,t,z)} :  \bR^d \times B_d(0, 4M + 1) \to \bR^{d \times d} $ by
\begin{align*}
G_h(x,y) = b^{-1} \left[ D^2h(s b^{-1} (  x + t y ) + z) - D^2h(sb^{-1}x + z) \right] b^{-1},
\end{align*}
where $D^2h(x) = [\partial_{\alpha} \partial_\beta h(x) ]_{\alpha,\beta}$. For a differentiable function $F : \bR^d \times B_d(0, 4M + 1) \to \bR^{d \times d}$ we set 
\begin{align*}
\Vert F  \Vert_\infty =  \sup \{ \Vert F(x,y) \Vert_s \, : \, (x,y)  \in \bR^d \times B_d(0,4M+1)  \}  
\end{align*}
and
\begin{align*}
\Vert \nabla F \Vert_\infty =  \max_{1 \le i \le 2d} \sup \{ \Vert \partial_i F(x,y) \Vert_s \, : \, (x,y)  \in \bR^d \times B_d(0,4M+1) \},
\end{align*}
where $\partial_i F(x, y) = [ \partial_i F_{\alpha,\beta}(x, y) ]_{\alpha,\beta}$.

Here is the first main result:

\begin{thm}\label{thm:main_multi_general} Fix $N \ge 1$ and let \(h: \, \bR^d \to \bR\) be three times differentiable with \(\Vert D^p h \Vert_{\infty} < \infty\) for \(1 \le p \le 3\). Suppose $M = \max_{i < N} \Vert g^i \Vert_{\infty} < \infty$, and that there exist a function \(\rho : \bN \to \bR_{+}\) with $\lim_{n \to \infty} \rho(n) = 0$ and constants $C_i > 0$, $1 \le i \le 3$, such that the following conditions hold for all \(0 \le n,m \le N-1\):

\begin{itemize}
\item[(A1)]{For all $\alpha, \beta \in \{1, \ldots , d \}$,
\begin{align*}
| \mu( \bar{f}^n_{\alpha} \bar{f}^{m}_{\beta} ) | \le C_1 \rho(|n-m|).
\end{align*}}
\item[(A2)]{Whenever \( (s,t,z) \in [0,1]^2 \times \bR^d \) and $m \le k \le N-1$,
\begin{align*}
\left| \mu \left[  ( \bar{f}^{n})^T   G_h \left( \sum_{ 0 \le i \le N-1 : \, |i - n| > k} \bar{f}^i, \bar{f}^{n,k}   \right)  \bar{f}^{n,m}   \right] \right| \le C_2 ( \Vert G_h \Vert_\infty + \Vert \nabla G_h \Vert_\infty ) \rho (m)  .
\end{align*}
} 
\item[(A3)]{Whenever $(s, t ,z) \in [0,1]^2 \times \bR^d$ and  \(2m \le k \le N -1\),
\begin{align*}
\left| \mu \left[  ( \bar{f}^{n})^T   \overline{G_h \left( \sum_{0 \le i \le N-1 : \, |i - n| > k} \bar{f}^i, \bar{f}^{n,k}   \right)}  \bar{f}^{n,m}   \right] \right|  \le C_3 ( \Vert G_h \Vert_\infty + \Vert \nabla G_h \Vert_\infty )  \rho(k -m).
\end{align*}
} 
\item[(A4)]{The matrix $\Sigma = \mu(W \otimes W)$ is positive definite.}\smallskip
\end{itemize}
Then
\begin{align}\label{eq:main_bound}
|\mu[h(W)] - \Phi_{\Sigma}(h)| \le C_* \Vert D^3 h \Vert_{\infty} N  \Vert b^{-1} \Vert^{3}_s  \sum_{m=1}^{N-1} m\rho(m), 
\end{align}
where 
\begin{align*}
C_* =  M^3 d^4 10 ( C_1 + C_2 + C_3 + 4  ).
\end{align*}
Here \(\Phi_{\Sigma}(h)\) denotes the expectation of \(h\) with respect to $\cN(0,\Sigma)$.
\end{thm}

\begin{remark} Theorem \ref{thm:main_multi_general}, as well as Theorems \ref{thm:main_1d}  and \ref{thm:main_self} given below, continue to hold if $f^i$ are replaced with general random vectors. 
\end{remark}

We postpone proving Theorem \ref{thm:main_multi_general} and other results in this section until Section \ref{sec:proof}. Due to the smooth metric, the constant $C_*$ in the upper bound \eqref{eq:main_bound} is independent of the covariance matrix $\Sigma$. Note that under the additional assumptions $\sum_{m=1}^{\infty} m \rho (m) < \infty$ and $ \Vert b^{-1} \Vert_s \lesssim N^{-1/2}$ we obtain $|\mu[h(W)] - \Phi_{\Sigma}(h)| = O(N^{-1/2})$ as $N \to \infty$, which is the optimal rate in this generality. Conditions (A1)-(A3) are designed for time-dependent systems with sufficiently good (polynomial) mixing properties.  Condition (A1) requires the decay of non-stationary correlations at the rate $\rho$. Condition (A2) requires that, for large $m$, the random vectors
\beqn
\bar{f}^n \hspace{0.5cm} \text{and} \hspace{0.5cm}  G_h \left( \sum_{0 \le i \le N-1 : \, |i - n| > k}  \bar{f}^i, \bar{f}^{n,k} \right)\bar{f}^{n,m}  
\eeqn
are componentwise nearly uncorrelated. This is reasonable because the function on the right depends on $\bar{f}^i$ with $|i - n| \ge m$ only. The function $G_h$ is differentiable and its $C^1$ norm appears as a factor in the upper bound. Condition (A3) is similar in spirit to condition (A2), for it requires
\begin{align*}
G_h \left( \sum_{0 \le i \le N-1 : \, |i - n| > k} \bar{f}^i, \bar{f}^{n,k} \right)\hspace{0.5cm} \text{and} \hspace{0.5cm} \bar{f}^{n} \otimes \bar{f}^{n,m} 
\end{align*}
to be nearly componentwise uncorrelated, which is again reasonable when ~$k \gg m$.

Recall that the Wasserstein distance between two random vectors $Y_1$ and $Y_2$  is defined by
\beqn
d_\mathscr{W}(Y_1,Y_2) = \sup_{h\in\mathscr{W}}|\mu(h(Y_1)) - \mu(h(Y_2))|,
\eeqn
where
\beqn
\mathscr{W} = \{h:\bR^d\to\bR\,:\,|h(x) - h(y)| \le \Vert x-y \Vert \}
\eeqn
is the class of all $1$-Lipschitz functions. When $d=1$ we obtain a result similar to Theorem \ref{thm:main_multi_general} for the Wasserstein distance. The relaxed smoothness of $h$ comes with the expense that conditions (A2) and (A3) have to be verified for a whole class of regular functions.

For a function $G : \bR^d \to \bR$ we denote
\begin{align*}
\text{Lip}(G) =  \sup_{x \neq y} \frac{| G(x) - G(y) | }{\Vert x - y \Vert}.
\end{align*}

\begin{thm}\label{thm:main_1d}  Let $d = 1$ and fix $N \ge 1$. Take $b = \textnormal{Var}_\mu( \sum_{i < N} \bar{f}^i )^{1/2}$. Suppose that $M = \max_{i < N} \Vert g^i \Vert_{\infty} < \infty$, that $b > 0$, and that there exist constants $C_i > 0$, $1 \le i \le 3$, and a function \(\rho : \bN \to \bR_{+}\) with $\lim_{n \to \infty} \rho(n) = 0$ such that the following conditions hold for all \(0 \le n,m \le N-1\):

\begin{itemize}
\item[(B1)]{$| \mu( \bar{f}^n \bar{f}^{m} ) | \le C_1 \rho(|n-m|)$.
} \medskip
\item[(B2)]{Whenever $m \le k \le N-1$ and $G : \bR \times B_1(0,4M+1) \to \bR$ is a bounded Lipschitz continuous function,
\begin{align*}
\left| \mu \left[   \bar{f}^{n}   G \left( \sum_{0 \le i \le N-1 : \, |i - n| > k}  \bar{f}^i, \bar{f}^{n,k}   \right)  \bar{f}^{n,m}   \right] \right| \le C_2 ( \Vert G \Vert_\infty +\textnormal{Lip}(G) ) \rho (m)  .
\end{align*}
}
\item[(B3)]{Whenever $2m \le k \le N-1$ and $G : \bR \times B_1(0,4M+1) \to \bR$ is a bounded Lipschitz continuous function,
\begin{align*}
\left| \mu \left[   \bar{f}^{n}  \bar{f}^{n,m}    \overline{G \left( \sum_{0 \le i \le N-1 : \, |i - n| > k} \bar{f}^i, \bar{f}^{n,k}   \right)}   \right] \right|  \le C_3 ( \Vert G \Vert_\infty + \textnormal{Lip}(G))  \rho(k -m).
\end{align*}}
\end{itemize}
Then 
\begin{align}\label{eq:scalar_bound}
& d_\mathscr{W}(W,Z) \le C_* N b^{-3}  \sum_{m=1}^{N-1}m \rho(m),
\end{align}
where 
\begin{align*}
C_* = 96 M^3 \biggl(   C_1 + C_2  
+ C_3    +  1   \biggr),
\end{align*}
and $Z \sim \cN(0,1)$ is a random variable with standard normal distribution.
\end{thm}

The following easy observation allows for normalizing constants other than 
\begin{align*}
b = \left[ \text{Var}_\mu \left( \sum_{i =0}^{N-1} \bar{f}^i \right) \right]^{1/2}.
\end{align*}
\begin{lem}\label{lem:1d_change_normal}
Suppose~\eqref{eq:scalar_bound} of Theorem~\ref{thm:main_1d} holds. Then, for any $c>0$,
\begin{align}\label{eq:1d_general_scale}
d_\mathscr{W} \left(c^{-1} \sum_{i=0}^{N-1} \bar{f}^i  ,c^{-1}bZ\right) \le C_* Nc^{-1}b^{-2} \sum_{m=1}^{N-1}m \rho(m).
\end{align}
\end{lem}
\begin{proof}
For any random variables $X,Y$ and any $a>0$, the Wasserstein metric satisfies
\begin{align*}
d_\mathscr{W}(aX,aY) = a\, d_\mathscr{W}(X,Y). 
\end{align*}
\end{proof}

\begin{remark} There is a notable difference between the upper bounds \eqref{eq:main_bound}  and \eqref{eq:1d_general_scale}: unlike \eqref{eq:main_bound},  \eqref{eq:1d_general_scale} always depends on $\textnormal{Var}_\mu( \sum_{i < N} \bar{f}^i )$ in addition to the normalizing constant $c$. This difference is due to the choice of metric.
\end{remark}

\subsection{Self-normalization}\label{sec:self_rate} We now assume that $ \operatorname{Cov}_\mu (\sum_{i=0}^{N-1} \bar f^i )$ is positive definite and set $b = \operatorname{Cov}_\mu (\sum_{i=0}^{N-1} \bar f^i )^{1/2}$ so that $\Sigma = \mu(W \otimes W) = I_{d \times d}$. In this case we establish an upper  bound on the distance between the law of $W$ and a standard normal random vector $Z \sim \cN(0, I_{d\times d})$ with respect to the Wasserstein metric. Unlike Theorem \ref{thm:main_1d},  the result applies for a general $d \ge 1$.
We denote by $\lambda_{\min}$ the least eigenvalue of $\operatorname{Cov}_\mu (\sum_{i=0}^{N-1} \bar f^i )$.

\begin{thm}\label{thm:main_self} Let $N \ge 1$. Suppose that $\max_{0 \le i < N} \Vert g^i \Vert_{\infty} \le M$ where $M \ge 1$, that $\lambda_{\min} > 1$, and that there exist a non-increasing function \(\rho : \bN \to \bR_{+}\) with with $\lim_{n \to \infty} \rho(n) = 0$ and constants $C_i > 0$, $1 \le i \le 3$, such that the following conditions hold for all  $0 \le n,m \le N-1$:
\begin{itemize}
\item[(C1)]{For all $\alpha, \beta \in \{1, \ldots , d \}$,
\begin{align*}
| \mu( \bar{f}^n_{\alpha} \bar{f}^{m}_{\beta} ) | \le C_1 \rho(|n-m|).
\end{align*}
}
\item[(C2)]{Whenever $m \le k \le N-1$ and $G : \bR^d \times B_d(0,4M+1) \to \bR^{d \times d}$ is a bounded $C^1$-function with bounded gradient,
\begin{align*}
\left| \mu \left[  ( \bar{f}^{n})^T   G \left( \sum_{0 \le i \le N-1 : \, |i - n| > k} \bar{f}^i, \bar{f}^{n,k}   \right)  \bar{f}^{n,m}   \right] \right| \le C_2 ( \Vert G \Vert_\infty + \Vert \nabla G \Vert_\infty ) \rho (m)  .
\end{align*}
}
\item[(C3)]{Whenever $2m \le k \le N-1$ and $G : \bR^d \times B_d(0,4M+1) \to \bR^{d \times d}$ is a bounded $C^1$-function with bounded gradient,
\begin{align*}
\left| \mu \left[  ( \bar{f}^{n})^T   \overline{G \left( \sum_{0 \le i \le N-1 : \, |i - n| > k} \bar{f}^i, \bar{f}^{n,k}   \right)}  \bar{f}^{n,m}   \right] \right|  \le C_3 ( \Vert G \Vert_\infty + \Vert \nabla G \Vert_\infty )  \rho(k -m).
\end{align*}}
\end{itemize}
Then
\begin{align*}
d_{\mathscr{W}}(W,Z) \le C_* N  (1 + \log N) \lambda_{\min}^{-\frac32} \sum_{m=1}^{N-1}  (1 + \log (\rho(m)^{-1} ) ) m \rho ( m ),
\end{align*}
where
\begin{align*}
C_* &= \biggl(  \frac{8 + 16d}{d} \frac{\Gamma( \frac{1 + d}{2})}{\Gamma(\frac{d}{2})} + \sqrt{d}4M + 2  \biggr) 2019 d^2 4^d M^4 \\
&\times \left[ (C_2 + C_3)(1 + \rho(0)^{-1})( 2\rho(0) + 1 ) + C_1 + 1 \right].
\end{align*}
\end{thm}
\medskip

\begin{remark} If in addition $\sum_{m=1}^{\infty} (1 + \log (\rho(m)^{-1} ))m \rho(m) < \infty$ and $\lambda_{\min} \gtrsim N$ hold, then we obtain the rate $d_{\mathscr{W}}(W,Z)= O(N^{-1/2} \log N)$, as $N \to \infty$. 
\end{remark}

\subsection{P{\`e}ne's CLT for stationary dynamics}\label{sec:pene} The theorems given above apply in the stationary case where $T_n = T$ preserves the measure $\mu$ for all $n \ge 1$. In this case the problem of normal approximation has been studied in several important articles including \cite{gouezel2005, Dedecker_2008, Liverani_2018, Pene_2002, Dubois_2011, Jan_2000}, using different  methods, conditions, and metrics. In the multidimensional case $d > 1$,  P{\`e}ne \cite{pene2005} formulated a correlation-decay condition for stationary processes, based on the inductive proof of Rio \cite{Rio_1996}. Let $S_n = \sum_{i=0}^{n-1}  f^i $, where $f^i = f \circ T^i$, $f : X \to \bR^d$ is bounded and $\mu(f) = 0$. In this context of measure preserving transformations, P{\`e}ne's condition can be stated as follows:

\begin{itemize}
\item[(D)]{There exist $r \in \mathbb{Z}_+$,  $C > 0$, $M \ge \max\{ 1 , \Vert f \Vert_\infty \}$, and a sequence of real numbers $(\varphi_{p,l})$ with $|\varphi_{p,l}| \le 1$ and $ \sum_{p=1}^{\infty} p \max_{0 \le l \le \round{p/(r+1)}} \varphi_{p,l} < \infty $, such that for any integers \(a,b,c \ge 0\) satisfying \(1 \le a + b + c \le 3\), for any integers \(i,j,k,p,q,l\) with \(0 \le i \le j \le k \le k + p \le k + p + q \le k + p + l\), for any \(\alpha, \beta, \gamma \in \{1,\ldots , d\}\), and for any bounded differentiable function \(F : \, \bR^d \times ([-M,M]^d)^3 \to \bR\) with bounded gradient,
\begin{align*}
&| \textnormal{Cov}_\mu[ F(S_{i},f^i,f^j,f^k),(f^{k+p}_{\alpha})^a(f^{k+p+q}_{\beta})^b(f^{k+p+l}_{\gamma})^c] | \notag \\
&\le C(\Vert F \Vert_{\infty} + \Vert \nabla F \Vert_{\infty})\, \varphi_{p,l}.
\end{align*}}
\end{itemize}

Condition (D) is satisfied by chaotic dynamical systems such as Sinai billiards. It was shown in \cite{pene2005} that condition (D) implies the existence of the limit $\Sigma_{0} := \lim_{n\to\infty} n^{-1} \mu(S_n \otimes S_n) $ and,  whenever $\Sigma_{0}$ is nonnull, the existence of a constant $B > 0$ such that
\begin{align}\label{eq:pene}
d_{\mathscr{W}}( N^{-\frac12} S_N, U) \le B N^{-\frac12} \qquad \forall N \ge 1,
\end{align}
where $U$ is a Gaussian random variable with expectation $0$ and covariance matrix $\Sigma_0$.

Compared to Theorem \ref{thm:main_multi_general}, \eqref{eq:pene} gives an upper bound of the same order $O(N^{-1/2})$ for stationary systems whose correlations decay at a rate which has a finite first moment, for test functions that are only assumed to be Lipschitz continuous. On the other hand, Theorem  \ref{thm:main_multi_general} is more general in that it applies for rather arbitrary matrix-valued normalizing sequences $b(N)$. Furthermore, the constant $C_*$ in \eqref{eq:main_bound} is more explicit than the one in \eqref{eq:pene} in terms of its dependence on $d$, $f$ and the underlying dynamical system. The same can be said about the constant $C_*$ in Theorem \ref{thm:main_self}, which gives an upper bound for the same metric as \eqref{eq:pene} but with a slightly weaker rate of convergence due to the logarithmic factor. Note that, similarly to conditions (B2)-(B3) and (C2)-(C3), condition (D) has to be verified for a whole class of regular functions $F$. 

\section{Application I: Random $1D$ piecewise expanding maps}\label{sec:pw}

In this section we apply Theorem \ref{thm:main_1d} to estimate the rate of convergence in the quenched CLT for a class of piecewise expanding random dynamical systems. Namely we consider the setup studied by Dragi\v{c}evi\'{c} et al. in \cite{Froyland_2018}. Below we recall some definitions and results from \cite{Froyland_2018} as they are necessary for understanding the application given in this section.

Set $(X,\cB) = ([0,1], \text{Borel}([0,1]) )$ and for a function $g : X \to \bR$ define its total variation by
\begin{align*}
V(g) = \inf_{h = g \text{ $m$-a.e}} \sup_{0 = x_0 < \ldots < x_n =1} \sum_{k=1}^n |h(x_k) - h(x_{k-1}) |.
\end{align*}
Moreover, define
\begin{align*}
\Vert g \Vert_{\text{BV}} = V(g) + \Vert g \Vert_{L^1(m)}.
\end{align*}
The Banach space $BV$ consists of all functions $g$ with $V(g) < \infty$ and is equipped with the norm $\Vert \cdot \Vert_{BV}$.

Let us denote by $\cE$ the collection of all maps $T : X \to X$ for which there exists a finite partition $\cA(T)$ of $X$ into subintervals such that for every $I \in \cA(T)$:
\begin{itemize}
\item[(1)]{$T \upharpoonright I$ extends to a $C^2$ map in a neighborhood of $I$;} \smallskip
\item[(2)]{$\delta(T) := \inf |T'| > 1$.}
\end{itemize}
The map $T$ is monotonous on each element $I \in \cA(T)$. From now on we take $\cA(T)$ to be the minimal such partition and set $N(T) = | \cA(T) |$.

Let $(\Omega, \cF, \bP)$ be a probability space and let $\tau : \Omega \to \Omega$ be an invertible $\bP$-preserving transformation. We consider a map $\omega \mapsto T_{\omega}$ from $\Omega$ into $\cE$. Random compositions of maps are denoted by
\begin{align*}
T^n_{\omega} = T_{\tau^{n-1}\omega} \circ \cdots \circ T_{\omega}
\end{align*}
and
\begin{align*}
\cL^n_{\omega} = \cL_{\tau^{n-1}\omega} \cdots \cL_{\omega},
\end{align*}
where $\cL_\omega : L^1(m) \to L^1(m)$ is the transfer operator associated to $T_\omega$:
\begin{align*}
\cL_\omega f (x) = \sum_{T_\omega(y) = x} \frac{ f( y )   }{| (T_\omega)'y |}.
\end{align*}

\noindent\textbf{Conditions (H):}

\begin{itemize}
\item[(i)]{$\tau : \Omega \to \Omega$ is invertible, $\bP$-preserving, and ergodic.} \smallskip
\item[(ii)]{The map $(\omega, x) \mapsto (\cL_{\omega} H(\omega, \cdot) )(x)$ is measurable for every measurable function $H : \Omega \times X \to \bR$ such that $H(\omega, \cdot) \in L^1(X,m)$. } \smallskip
\item[(iii)]{$N := \sup_{\omega \in \Omega} N(T_\omega) < \infty$; $\delta : =  \inf_{\omega \in \Omega} \delta(T_\omega) > 1$; $D := \sup_{\omega \in \Omega} |T''_\omega| < \infty$.
} \smallskip
\item[(iv)]{There is $r \ge 1$ such that $\delta^r > 2$ and $\text{ess inf}_{\omega \in \Omega} \min_{J \in \cA(T_\omega^r)} m(J) > 0$.} \smallskip
\item[(v)]{For every subinterval $J \subset X$ there is $k = k(J) \ge 1$ such that $T^k_\omega(J) = X$ holds for almost every $\omega \in \Omega$.}
\end{itemize}

\begin{remark} It was shown in \cite{Froyland_2018} that conditions (H) imply several nice properties for the transfer operators $\cL_{\omega}$, including a Lasota-Yorke inequality and exponential decay in the BV-norm. The authors used such properties to establish an almost sure invariance principle.
\end{remark}

\begin{lem}[See Proposition 1 in \cite{Froyland_2018}]\label{lem:froy_1} Assume conditions (H). Then there exists a unique measurable and non-negative function $h : \Omega \times X \to \bR$ such that $h_\omega := h(\omega,\cdot) \in BV$, $m(h_\omega) =1$ and $\cL_\omega(h_\omega) = h_{\tau(\omega)}$ for almost every $\omega \in \Omega$. Moreover, $\textnormal{ess sup}_{\omega \in \Omega} \Vert h_{\omega} \Vert_{BV} < \infty$.
\end{lem}

\subsection{Statement of result} Let $f : X \to \bR$ be a bounded measurable function and set
\begin{align*}
 \widetilde{f}(\omega,x) = \widetilde{f}_\omega(x) = f(x) - \mu_{\omega}(f),
\end{align*}
where $d \mu_{\omega} = h_\omega \, dm$ and $h$ is the function from Lemma \ref{lem:froy_1}. Set
\begin{align*}
W(\omega) = b^{-1} \sum_{n=0}^{N-1} \widetilde{f}_{\tau^n(\omega)} \circ T^n_{\omega} = b^{-1} \sum_{n=0}^{N-1} (f \circ T^n_{\omega} - \mu_\omega(f \circ T^n_\omega)),
\end{align*}
where $b$ is the square root of  ${\mu_\omega} [ (  \sum_{n=0}^{N-1} \widetilde{f}_{\tau^n(\omega)} \circ T^n_{\omega})^{2} ] $. We denote by $\varphi : \Omega \times X \to \Omega \times X$ the skew product $\varphi(\omega, x) = (\tau(\omega), T_{\omega}(x))$, which preserves the measure $\mu$ on $\Omega \times X$ defined by 
\begin{align*}
\mu(A \times B) = \int_{A \times B}  h \, d ( \bP \times m), \hspace{0.5cm} A \in \cF, \, B \in \cB.
\end{align*}

\begin{thm}\label{thm:pw} Consider a family of piecewise expanding maps $(T_{\omega})_{\omega \in \Omega}$ such that conditions (H) hold. Fix $N \ge 1$ and suppose $f$ is Lipschitz continuous such that $\widetilde{f}$ can not be written as $g - g \circ \varphi$ for any $g \in L^2(\Omega \times X, \mu)$. Then there is $C_* > 0$ independent of $N$ such that
\begin{align*}
d_{\mathscr{W}}(W,Z) \le C_* N^{-\frac12}
\end{align*}
holds for almost every $\omega \in \Omega$. Here $Z \sim \cN(0,1)$ is a random variable with standard normal distribution.
\end{thm}

\begin{remark} The proof of Theorem \ref{thm:pw} is based on Theorem \ref{thm:main_1d}. Theorem \ref{thm:main_multi_general} or \ref{thm:main_self} could be used instead to obtain similar central limit theorems for multivariate observables $f : X \to \bR^d$. \end{remark}

\subsection{A functional correlation bound} Conditions (B2) and (B3) of Theorem \ref{thm:main_1d} will be verified by applying the auxiliary result given below, which facilitates bounding integrals of the form $\int F \circ (T_\omega^m )_{0 \le m < k} \, d\mu$, where $F : [0,1]^k \to \bR$ is not necessarily a product of one-dimensional observables. Such functional correlation bounds were established for stationary Sinai billiards in \cite{LS_2017} and for time-dependent intermittent maps in \cite{Leppanen_2017}.  

For a function $F : [0,1]^k \to \bR$, $\theta \in (0,1]$, and $1 \le \beta \le k$ we denote 
\begin{align}\label{eq:holder}
 [F]_{\theta,\beta} = \sup_{x \in [0,1]^k} \sup_{a \neq a'}  \frac{ | F(x(a/\beta)  ) - F(x(a'/\beta) )|  }{|a - a' |^{\theta}},
\end{align}
where $x(a/\beta) \in [0,1]^k$ is obtained from $x$ by replacing the $\beta$th coordinate with $a \in [0,1]$. We say that $F$ is $\theta$-Hölder continuous in the coordinate $\beta$ if $[F]_{\theta,\beta} < \infty$.

\begin{prop}\label{lem:pw_funct_corr} Let $k \ge 2$. Consider integers  \(0 \le n_1 \le \ldots \le n_k\) blocked according to a set of indices \(0 = \ell_0 < \ell_1 < \ldots < \ell_p < \ell_{p+1} = k\), where we assume that $n_{\ell_i + 1} < \ldots < n_{\ell_{i+1}}$ hold for all $0 \le i \le p$. Suppose $(T_\omega)_{\omega\in\Omega} \subset \cE$ is a family of maps such that conditions (H) hold, and that \(F_{\omega} : \, [0,1]^{k} \to \bR\) is a function with $\textnormal{ess sup}_{\omega\in \Omega} \Vert F_\omega \Vert_{\infty} < 0$ and
\begin{align*}
L = \textnormal{ess sup}_{\omega \in \Omega} \sup_{1 \le \beta \le \ell_p } [F_\omega]_{\theta,\beta} < \infty.
\end{align*}
Denote by \(H_{\omega}(x_1,\ldots, x_{p+1})\) the function
\begin{align*}
F_\omega(T^{n_1}_{\omega} (x_1),\ldots, T^{n_{\ell_1}}_{\omega} (x_1), T^{n_{{\ell_1}+1}}_\omega (x_2),\ldots, T^{n_{{\ell_2}}}_\omega (x_2), \ldots, T^{n_{\ell_p+1}}_\omega (x_{p+1}),\ldots, T^{n_k}_\omega (x_{p+1})).
\end{align*}
Then, for any probability measures \(\mu_1,\ldots,\mu_{p+1}\) whose densities belong to \(BV\), and for almost every $\omega \in \Omega$, 
\begin{align}
&\left| \int H_\omega(x,\ldots,x) \, d\mu_1(x) - \idotsint H_\omega(x_1,\ldots,x_{p+1}) \, d\mu_1(x_1)d\mu_2(x_2)\ldots \, d\mu_{p+1} (x_{p+1}) \right| \notag\\ 
&\le C (  \textnormal{ess sup}_{\omega \in \Omega} \Vert F_\omega \Vert_{\infty} + L ) \left( \sum_{i=1}^{p+1} \Vert h_i \Vert_{BV} \right)  \sum_{i=1}^p \gamma^{n_{\ell_i + 1} - n_{\ell_i}}  \label{eq:pw_funct_corr},
\end{align}
where $0 < \gamma < 1$, and $C = C(p, (T_\omega)_{\omega \in \Omega}, \theta) > 0$.

\end{prop}

\begin{remark} The upper bound \eqref{eq:pw_funct_corr} is independent of $k$.

\end{remark}

The proof for Proposition \ref{lem:pw_funct_corr} is based on two auxiliary results. The first result is an immediate consequence of Corollary 8  in \cite{rousseau2016} due to Aimino and Rousseau, who considered sequential (non-random) compositions of piecewise-expanding maps. The second result is Lemma 2 in the paper \cite{Froyland_2018} by Dragi\v{c}evi\'{c} et al.

\begin{lem}\label{lem:ar} Suppose conditions (H) hold. There is $C > 0$ such that for almost every $\omega \in \Omega$,
\begin{align}\label{eq:ar}
\sum_{I \in \cA(T_\omega^n)} V_I \left( \frac{1}{|(T^n_\omega)' |} \right) \le C,
\end{align}
where $V_I(f)$ denotes the total variation of $f$ over the subinterval $I \subset [0,1]$.
\end{lem}

\begin{proof} As is explained on p. 2252 of \cite{Froyland_2018}, condition (iv) implies that there exists $\alpha^r \in (0,1)$ and $K^r  > 0$ such that, for almost every $\omega \in \Omega$, 
\begin{align}\label{eq:lasota_yorke}
V(\cL_{\omega}^{\ell r} \phi ) \le ( \alpha^r )^\ell V(\phi) + \frac{K^r}{1 - \alpha^r} \Vert \phi \Vert_{L^1(m)}
\end{align}
holds for all $\phi \in BV$ and $\ell \ge 1$. It suffices to fix $\Omega^* \subset \Omega$ with $\bP(\Omega^*) = 1$ such that \eqref{eq:lasota_yorke} holds for all $\omega \in \Omega^*$. Then the proof of Corollary 8 in \cite{rousseau2016} shows that \eqref{eq:ar} holds for all $\omega \in \Omega^*$.
\end{proof}

\begin{lem}[See Lemma 2 in \cite{Froyland_2018}]\label{eq:memory_loss} Assume conditions (H). There is $K > 0$ and $\eta \in (0,1)$ such that, for almost every $\omega \in \Omega$,
\begin{align*}
\Vert \cL_\omega^n \phi \Vert_{BV} \le K \eta^n \Vert \phi \Vert_{BV}
\end{align*}
holds for all $n \ge 0$ and $\phi \in BV$ with $m(\phi) = 0$.

\end{lem}

\begin{proof}[Proof for Proposition \ref{lem:pw_funct_corr}] The proof proceeds by induction on $p$. First let $p = 1$ and denote $\ell_1 = \ell$. Then the function $H_\omega(x,y)$ in Proposition \ref{lem:pw_funct_corr} becomes
\begin{align*}
H_\omega(x,y) = F_\omega( T^{n_1}_\omega(x), \ldots, T^{n_\ell}_\omega(x), T^{n_{\ell + 1}}_\omega(y), \ldots, T^{n_k}_\omega(y) ),
\end{align*}
where $n_{1} < \ldots < n_{\ell} \le n_{\ell + 1} < \ldots < n_k$.  Set \(n_* = n_\ell + \round{(n_{\ell+1} - n_\ell)/2}\)\footnote{We denote by \(\round{x}\) the greatest non-negative integer \(n\) with \(n \le x\).}. Then,
\begin{align*}
\int H_\omega(x,x) \, d \mu_1(x) &= \iint H_\omega(x,y) \, d\mu_1(x) \, d\mu_2(y)  \\
&=  \sum_{J \in \cA(T_\omega^{n_*} ) } \int_{J} \left( H(x,x) - \int H(x,y) \, d\mu_2(y) \right) \, d\mu_1(x).
\end{align*}

\noindent\textbf{Claim.}  If \(a,b \in J \in  \cA(T_\omega^{n_*}) \), then almost surely
\begin{align}\label{eq:funct_pw_claim}
&\left|H_\omega(a,x) - \int H_\omega(a,y) \, d\mu_2(y) - \left(H_\omega(b,x) - \int H_\omega(b,y) \, d\mu_2(y)\right) \right| 
\le L_1 C \kappa^{n_{\ell + 1} - n_{\ell}},
\end{align}
where \(L_1 = \text{ess sup}_{\omega \in \Omega} \max_{1 \le\alpha \le \ell} [F_\omega]_{\theta,\alpha} \), $\kappa = (\delta^{\theta})^{-1/2} \in (0,1)$, and $C = C(\kappa) > 0$. We recall that by definition  $\delta = \inf_{\omega \in \Omega} \delta(T_\omega) > 1$.

\begin{proof}[Proof for Claim] Since $F_\omega$ is $\theta$-Hölder continuous for a.e. $\omega \in \Omega$ in the first $\ell$ coordinates,
\begin{align*}
|F_\omega(a_1,\ldots, a_{\ell}, x,\ldots, x) - F_\omega(b_1,\ldots, b_{\ell}, x,\ldots, x)| \le L_1 \sum_{\alpha =1}^{\ell} |a_{\alpha} - b_{\beta}|^{\theta}
\end{align*}
holds for a.e. $\omega \in \Omega$. Consequently,
\begin{align*}
&|H_\omega(a,x) - H_\omega(b,x) | \le L_1 \sum_{\alpha =1}^{\ell} | T^{n_\alpha}_\omega(a) - T^{n_\alpha}_\omega (b)|^{\theta} \leq L_1 \sum_{\alpha=1}^{\ell} m(T^{n_\alpha}_\omega J)^\theta . 
\end{align*}

For each $1 \le \alpha \le \ell$,  $T_{\tau^{n_\alpha}\omega}^{n_* - n_\alpha}$ maps $T_\omega^{n_\alpha}(J)$ diffeomorphically onto $T^{n_*}_\omega(J)$, which implies the upper bound $m(T^{n_\alpha}_\omega J) \le (\delta^{n_* - n_\alpha})^{-1}  m(T^{n_*}_\omega J) \le \delta^{-n_* + n_\alpha}$. That is, for a.e. $\omega \in \Omega$,
\begin{align*}
|H_\omega(a,x) - H_\omega(b,x) | \le \sum_{\alpha=1}^{\ell} (\delta^\theta)^{-n_* + n_\alpha} \le \sum_{k = n_* - n_{\ell}}^{\infty} \kappa^k \le  C(\kappa) \kappa^{n_{\ell + 1} - n_{\ell}}.
\end{align*}
This proves the claim.
\end{proof}

We fix a point $c_J \in J$ for each $J \in \cA(T^{n_*}_\omega)$. Then  \eqref{eq:funct_pw_claim} implies for a.e. $\omega \in \Omega$ the upper bound
\begin{align}
&\left| \sum_{J \in \cA(T^{n_*}_\omega) } \int_{J} \left( H_\omega(x,x) - \int H_\omega(x,y) \, d\mu_2(y) \right) \, d\mu_1(x) \right| \notag \\
&\le \left| \sum_{J \in \cA(T^{n_*}_\omega) } \int_{J} \left( H_\omega(c_{J},x) - \int H_\omega(c_{J},y) \, d\mu_1(y) \right) \, d\mu_2(x) \right| \label{eq:pw_funct_corr_eq1} \\
&+ L_1 C \kappa^{n_{\ell + 1} - n_{\ell}}. \notag
\end{align}
Let \(h_1 \in BV\) denote the density of \(\mu_1\), and let \(h_2 \in BV\) denote the density of \(\mu_2\). Moreover, let \(\widetilde{H}_\omega(c_{J},x)\) be the function that satisfies \(\widetilde{H}_\omega(c_{J}, T^{n_{\ell+1}}_\omega(x)) = H_\omega(c_J, x) \). Fix $J \in \cA(T^{n_*}_\omega)$. Then, for a.e. $\omega \in \Omega$,
\begin{align*}
&\left| \int_{J} \left( H_\omega(c_{J},x) - \int H_\omega(c_{J},y) \, d\mu_1(y) \right) d \mu_2(x) \right|  \\
&= \left| \int_0^1 \left( H_\omega(c_{J},x) - \int H_\omega(c_{J},y) \, d\mu_1(y) \right)( 1_J(x) h_2(x) - \mu_2(J)h_1(x)  ) \, dm(x) \right| \notag \\
&= \left| \int_0^1 \left( H_\omega(c_{J},x) - \int H_\omega(c_{J},y) \, d\mu_1(y) \right) \cL_{\omega}^{n_{\ell+1}} ( 1_J(x) h_2(x) - \mu_2(J)h_1(x)  ) \, dm(x) \right| \notag \\
&\le 2 \, \text{ess sup}_{\omega \in \Omega} \Vert F_\omega \Vert_{\infty} \Vert \cL_{\omega}^{n_{\ell+1}} ( 1_J h_2 - \mu_2(J)h_1 ) \Vert_{L^1(m)}. \notag
\end{align*}

Let $x \in [0,1]$. Since $J \in \cA(T^{n_*}_\omega)$, either $x \in T^{n_*}_\omega(J)$ and
\begin{align*}
\cL_{\omega}^{n_*} (1_J h_2)x = \frac{h_2( ( T^{n_*}_\omega \upharpoonright J )^{-1} x ) }{| (T^{n_*}_\omega)'( T^{n_*}_\omega \upharpoonright J )^{-1} x) |},
\end{align*}
or $\cL_{\omega}^{n_*} (1_J h_2)x = 0$. It follows easily from this and the strict monotonicity of $T^{n_*}_\omega \upharpoonright J$ that
\begin{align*}
V( \cL_{\omega}^{n_*} (1_J h_2) ) \le 2 \Vert h_2 \Vert_{\infty}  \sup_{y \in J} \frac{1}{|(T^{n_*}_\omega)' y|} + \Vert h \Vert_{\infty}  V_J \left( \frac{1}{|(T^{n_*}_\omega)' |} \right) +  V_J(h_2) \sup_{y \in J} \frac{1}{|(T^{n_*}_\omega)' y|}, 
\end{align*}
where 
\begin{align*}
\sup_{y \in J} \frac{1}{|(T^{n_*}_\omega)' y|} \le V_J \left( \frac{1}{|(T^{n_*}_\omega)' |} \right) +  \int_0^1 \frac{1}{|(T^{n_*}_\omega)'( T^{n_*}_\omega \upharpoonright J )^{-1} y|} \, dy \le V_J \left( \frac{1}{|(T^{n_*}_\omega)' |} \right) + m(J).
\end{align*}
We conclude that
\begin{align}\label{eq:var_1}
V( \cL_{\omega}^{n_*} (1_J h_2) ) \le 6 \Vert h_2 \Vert_{BV} \left(  V_J \left( \frac{1}{|(T^{n_*}_\omega)' |} \right)  + m(J) \right).
\end{align}
On the other hand there is $C > 0$ such that, for any $\phi \in BV$, $\sup_{k \ge 0} V(\cL^k_\omega(\phi)) \le C \Vert \phi \Vert_{BV}$ holds for almost every $\omega \in \Omega$. This follows from \eqref{eq:lasota_yorke} together with the fact that $\Vert \cL_\omega(\phi) \Vert_{BV} \le C \Vert \phi \Vert_{BV}$ for almost every $\omega \in \Omega$; see p. 2257 of \cite{Froyland_2018}. In particular, 
\begin{align}\label{eq:var_2}
V(  \cL_{\omega}^{n_*} (\mu_2(J) h_1) ) \le \mu_2(J) C \Vert h_1 \Vert_{BV} \hspace{0.5cm} \text{a.e. $\omega \in \Omega$.}
\end{align}

Next we combine Lemma \ref{eq:memory_loss}, \eqref{eq:var_1} and \eqref{eq:var_2} to obtain
\begin{align*}
&\Vert \cL_{\omega}^{n_{\ell+1}} ( 1_J h_2 - \mu_2(J)h_1 ) \Vert_{L^1(m)} \\
&\le K \eta^{n_{\ell + 1} - n_*} \Vert \cL_{\omega}^{n_*} ( 1_J h_2 - \mu_2(J)h_1 ) \Vert_{BV} \\
&\le C \eta_1^{n_{\ell + 1} - n_{\ell}} ( \Vert h_1 \Vert_{BV} + \Vert h_2 \Vert_{BV}  ) \left( V_J \left( \frac{1}{|(T^{n_*}_\omega)' |} \right)  + m(J) + \mu_2(J) \right),
\end{align*}
for a.e. $\omega \in \Omega$, where $\eta_1 \in (0,1)$. Then, by Lemma \ref{lem:ar},
\begin{align*}
\eqref{eq:pw_funct_corr_eq1} &\le \sum_{J \in \cA(T^{n_*}_\omega  )} 2 \, \text{ess sup}_{\omega \in \Omega} \Vert F_\omega \Vert_{\infty}  \Vert \cL_{\omega}^{n_{\ell+1}} ( 1_J h_2 - \mu_2(J)h_1 ) \Vert_{L^1(m)} \\
&\le \sum_{J \in \cA(T^{n_*}_\omega  )} \text{ess sup}_{\omega \in \Omega} \Vert F_\omega \Vert_{\infty} C \eta_1^{n_{\ell + 1} - n_{\ell}} ( \Vert h_1 \Vert_{BV} + \Vert h_2 \Vert_{BV}  ) \\
& \times \left( V_J \left( \frac{1}{|(T^{n_*}_\omega)' |} \right)  + m(J) + \mu_2(J) \right) \\
&\le C\, \text{ess sup}_{\omega \in \Omega} \Vert F_\omega \Vert_{\infty} C \eta_1^{n_{\ell + 1} - n_{\ell}} ( \Vert h_1 \Vert_{BV} + \Vert h_2 \Vert_{BV}  ),
\end{align*}
for a.e. $\omega \in \Omega$. Taking $\gamma = \max \{ \eta_1, \kappa \}$ completes the proof for the case $p = 1$.

Suppose that we have shown \eqref{eq:pw_funct_corr} for $p-1$, and fix integers \(0 = \ell_0 < \ell_1 < \ldots < \ell_p < \ell_{p+1} = k\) as in the proposition.  Recall that  \(H_{\omega}(x_1,\ldots, x_{p+1})\) denotes the function
\begin{align*}
F_\omega(T^{n_1}_{\omega} (x_1),\ldots, T^{n_{\ell_1}}_{\omega} (x_1), T^{n_{{l_1}+1}}_\omega (x_2),\ldots, T^{n_{{l_2}}}_\omega (x_2), \ldots, T^{n_{l_p+1}}_\omega (x_{p+1}),\ldots, T^{n_k}_\omega (x_{p+1})).
\end{align*}

From the case \(p=1\) we know that, for a.e. $\omega \in \Omega$,
\begin{align}
&\left| \int H_\omega (x,\ldots,x) \, d\mu_1(x)  - \iint H_\omega(x,\ldots x,x_{p+1}) \, d\mu_1 (x) \, d\mu_{p+1} (x_{p+1}) \right| \notag \\
&\le C (\text{ess sup}_{\omega \in \Omega} \Vert F_\omega \Vert_{\infty}  + L ) \left(  \Vert h_1 \Vert_{BV} + \Vert h_{p+1} \Vert_{BV}  \right)  \gamma^{n_{\ell_p + 1} - n_{\ell_p}}  \label{eq:est1}
\end{align}
where $h_i$ is the density of $\mu_i$.

 Next for each \(x_{p+1} \in [0,1]\), we apply the induction hypothesis to the function 
 \begin{align*}
 (y_1,\ldots,y_k) \mapsto F_\omega(y_1,\ldots,y_{\ell_p}, T^{n_{\ell_p+1}}_\omega (x_{p+1}),\ldots, T^{n_k}_\omega (x_{p+1})).
 \end{align*}
This implies for a.e. $\omega \in \Omega$ the upper bound
\begin{align}
&\left| \int H_\omega (x,\ldots x,x_{p+1}) \, d\mu_1 (x) - \idotsint H_\omega(x_1,\ldots,x_p,x_{p+1}) \, d\mu_1(x_1) \, \ldots d\mu_{p}(x_{p}) \right| \notag \\
&\le C (\text{ess sup}_{\omega \in \Omega} \Vert F_\omega \Vert_{\infty}  + L ) \left( \sum_{i=1}^{p} \Vert h_i \Vert_{BV} \right)  \sum_{i=1}^{p-1} \gamma^{n_{\ell_i + 1} - n_{\ell_i}} , \label{eq:est2}
\end{align}
for all \(x_{p+1} \in [0,1]\). Now, to complete the proof for Proposition \ref{lem:pw_funct_corr},  it suffices to combine \eqref{eq:est1} and \eqref{eq:est2}. \\
\end{proof}

\subsection{Proof for Theorem \ref{thm:pw}}  It was shown in \cite{Froyland_2018} that there exists a non-random $\sigma^2 \ge 0$ such that 
\begin{align*}
\sigma^2 = \lim_{n \to \infty} \mu_{\omega} \left[ \left( \frac{1}{\sqrt{n}} \sum_{k=0}^{n-1} \widetilde{f}_{\tau^k\omega} \circ T^k_\omega \right)^2 \right]
\end{align*}
for almost every $\omega \in \Omega$. Moreover, $\sigma^2 = 0$ if and only if there exists $g \in L^2(\Omega \times X, \mu)$ such that $\widetilde{f} = g - g \circ \varphi$. Hence, under our assumption there exists $C > 0$ and $n_0 \ge 1$ such that, for a.e. $\omega \in \Omega$,
\begin{align*}
b^2 = \mu_{\omega} \left[ \left( \sum_{k=0}^{n-1} \widetilde{f}_{\tau^k\omega} \circ T^k_\omega \right)^2 \right] \ge Cn 
\end{align*}
holds for all $n \ge n_0$. 

Next we show that, with $\mu_\omega$ as the initial measure, conditions (B1)-(B3) hold with $\rho(m) = \gamma^m$ for a.e. $\omega \in \Omega$, where $\gamma \in (0,1)$ is the same as in  Proposition \ref{lem:pw_funct_corr}. To this end recall that, by Lemma \ref{lem:froy_1}, the density $h_\omega$ of $\mu_\omega$ lies in $BV$ for a.e. $\omega \in \Omega$.

(B1): For brevity, we introduce the notation $\widetilde{f}_\omega^n = f \circ T_\omega^n - \mu_\omega(f \circ T_\omega^n)$. Taking $k = 1$, $p = 1$, $F_\omega(x,y) = f(x)f(y)$, and $\mu_1 = \mu_\omega = \mu_2$ in Proposition  \ref{lem:pw_funct_corr}  yields the upper bound 
\begin{align}\label{eq:b1_pw}
| \mu_\omega( \widetilde{f}_\omega^n   \widetilde{f}_\omega^m   ) | \le C \Vert f \Vert_{\text{Lip}} \text{ess sup}_{\omega \in \Omega} \Vert h_\omega \Vert_{BV} \gamma^{|n-m|},
\end{align}
for a.e. $\omega \in \Omega$.

(B2): Let $m \le k \le N-1$ and let $G : \bR \times B_1(0, 4\Vert f \Vert_{\infty} + 1 ) \to \bR$ be a bounded Lipschitz continuous function. We define $F_\omega(x_0, \ldots, x_{n-k}, x_{n-m}, x_n, x_{n+m}, x_{n+k}, \ldots,  x_{N-1} )$ by the formula
\begin{align*}
\psi_{n, \omega}(x_n) G\left(   \sum_{| i -n | > k } \psi_{i, \omega}(x_i),  \sum_{| i - n | = k} \psi_{i, \omega}(x_i)   \right) \sum_{| i - n | = m} \psi_{i, \omega}(x_i),
\end{align*}
where $\psi_{i,\omega}(x) = f(x) - \mu_\omega(f \circ T_\omega^i)$ and the summations are over $i$. Then
\begin{align*}
&\mu_\omega \left[  F_\omega(  T_\omega^0, \ldots, T_\omega^{n-k}, T_\omega^{n-m}, T_\omega^n, T_\omega^{n+m}, T_\omega^{n+k}, \ldots,  T_\omega^{N-1}) \right] \\
&=  \mu_\omega \left[  \widetilde{f}_\omega^n    G \left( \sum_{|i - n| > k}\widetilde{f}_\omega^i , \sum_{|i-n|=k} \widetilde{f}_\omega^i  \right) \sum_{|i-n|=m} \widetilde{f}_\omega^i     \right],
\end{align*}
which is the integral we need to control. It is easy to verify that $F_{\omega}$ is Lipschitz continuous with
\begin{align*}
\sup_{\omega \in \Omega} \Vert F_\omega \Vert_{\infty} \le  \Vert G \Vert_{\infty} 4 \Vert f \Vert^2_\infty
\end{align*}
and
\begin{align*}
\sup_{\omega \in \Omega} \sup_{\beta \in \cI} [F_\omega]_{1, \beta} \le 8 \Vert f \Vert_{\text{Lip}}^3 (  \Vert G \Vert_\infty + \text{Lip}(G) ),
\end{align*}
where $\cI = \{ 0 \le i \le N-1 \: : \: |i-n| \ge k \} \cup \{ 0 \le i \le N-1 \: : \: |i-n| = m \} \cup \{n\} $ is an indexing for the arguments of $F$. Observe that, since $\mu_\omega( \widetilde{f}_\omega^n) =0 $,
\begin{align*}
\iiint   F_\omega(  \cT_\omega^{ \le n-k}x, T_\omega^{n-m}x, T_\omega^ny, T_\omega^{n+m}z,   \cT^{\ge n + k}z )  \,  d \mu_\omega(x) \,  d \mu_\omega(y)  \, d \mu_\omega(z) = 0,
\end{align*}
where $\cT_\omega^{ \le n-k}x = ( T^0_\omega x, \ldots, T^{n-k}_\omega x )$ and $\cT_\omega^{ \ge n + k}z = ( T^{n+k}_\omega z, \ldots, T^{N-1}_\omega z )$. It follows by Proposition \ref{lem:pw_funct_corr} applied with $F_\omega$ and $p=2$  that, for a.e. $\omega \in \Omega$,
\begin{align*}
&\left| \mu_\omega \left[  \widetilde{f}_\omega^n    G \left( \sum_{|i - n| > k}\widetilde{f}_\omega^i , \sum_{|i-n|=k} \widetilde{f}_\omega^i  \right) \sum_{|i-n|=m} \widetilde{f}_\omega^i     \right] \right| \\
&\le C \left( \sup_{\omega \in \Omega} \Vert F_\omega \Vert_{\infty} +  \sup_{\omega \in \Omega} \sup_{\beta \in \cI } [F_\omega]_{1, \beta}   \right) \,  \text{ess sup}_{\omega \in \Omega} \Vert h_\omega \Vert_{BV} \gamma^{m} \\
&\le C \biggl(  \Vert G \Vert_{\infty} 4 \Vert f \Vert^2_\infty +  8 \Vert f \Vert_{\text{Lip}}^3 (  \Vert G \Vert + \text{Lip}(G)   ) \biggr) \,  \text{ess sup}_{\omega \in \Omega} \Vert h_\omega \Vert_{BV} \gamma^{m} \\
&\le C (  \Vert G \Vert_\infty + \text{Lip}(G) ) (  \Vert f \Vert^2_\infty +  \Vert f \Vert_{\text{Lip}}^3 )\, \text{ess sup}_{\omega \in \Omega} \Vert h_\omega \Vert_{BV} \gamma^{m}.
\end{align*}

(B3): This is obtained in the same way as condition (B2). Namely, whenever $2m \le k \le N-1$, applying Proposition \ref{lem:pw_funct_corr} with $p=2$ and the function
\begin{align*}
\psi_{n, \omega}(x_n) G_* \left(   \sum_{| i -n | > k } \psi_{i, \omega}(x_i),  \sum_{| i - n | = k} \psi_{i, \omega}(x_i)   \right) \sum_{| i - n | = m} \psi_{i, \omega}(x_i),
\end{align*}
where
\begin{align*}
 G_*(x,y) = G(x,y) -   \mu_\omega \left[ G \left( \sum_{|i - n| > k} \bar{f}^i, \bar{f}^{n,k}   \right) \right],
\end{align*}
implies for a.e. $\omega \in \Omega$ the upper bound
\begin{align*}
&\left| \mu_\omega \left[  \widetilde{f}_\omega^n   \sum_{|i-n|=m} \widetilde{f}_\omega^i   \,  G_1 \left( \sum_{|i - n| > k}\widetilde{f}_\omega^i , \sum_{|i-n|=k} \widetilde{f}_\omega^i  \right)   \right] \right| \\
&\le C (  \Vert G \Vert_\infty + \text{Lip}(G) ) (  \Vert f \Vert^2_\infty +   \Vert f \Vert_{\text{Lip}}^3 ) \, \text{ess sup}_{\omega \in \Omega} \Vert h_\omega \Vert_{BV} \gamma^{m}.
\end{align*}

Since $\sum_{m=1}^{\infty} m \gamma^m < \infty$, Theorem \ref{thm:pw} now follows by Theorem \ref{thm:main_1d}.\\

\begin{remark} Another example of a random dynamical system that satisfies the conditions of Theorem \ref{thm:main_1d}  is the Sinai Billiard of \cite{stenlund2014}, in which a scatterer configuration on the torus is randomly updated between consecutive collisions. The key technical lemmas necessary for obtaining an analog of Proposition \ref{lem:pw_funct_corr} were proven in  \cite{stenlund2014, SYZ_2013}, including a statistical memory loss starting from an initial measure supported on a single homogeneous local unstable manifold (Lemma 12 of \cite{stenlund2014}), and a tail estimate on the prevalence of short local unstable manifolds (Lemma 13 of \cite{stenlund2014}). The application would imply a rate of convergence in the annealed CLT but we will not treat it here.
\end{remark} 

\section{Application II: intermittent maps }\label{sec:int} 

Following \cite{liverani1999} we define for each \(\alpha \in (0,1)\) the map \(T_{\alpha} : [0,1] \to [0,1]\) by
\begin{align*}
T_{\alpha }(x) = \begin{cases} x(1+ 2^{\alpha }x^{\alpha}) & \forall x \in [0, 1/2),\\
2x-1 & \forall x \in [1/2,1].
 \end{cases} \hspace{0.5cm}
\end{align*} 

Associated to each map \(T_{\alpha}\) is its transfer operator \(\cL_{\alpha} : L^1(m) \to L^1(m)\) defined by
\begin{align*}
\cL_{\alpha} h(x) = \sum_{y \in T_{\alpha}^{-1}\{x\}} \frac{h(y)}{T'_{\alpha}(y)}.
\end{align*}
\indent We denote by \(d \hat{\mu}_{\alpha} = \hat{h}_{\alpha} dm\) the invariant absolutely continuous probability measure associated to ~\(T_{\alpha}\). It follows from \cite{liverani1999} that the density \(\hat{h}_{\alpha}\) belongs to the convex cone of functions
\beqn
\begin{split}
\cC_*(\alpha) = \{f\in C((0,1])\cap L^1\,:\, & \text{$f\ge 0$, $f$ decreasing,} 
\\
& \text{$x^{\alpha+1}f$ increasing, $f(x)\le 2^{\alpha} (2 + \alpha) x^{-\alpha} m(f)$}\}.
\end{split}
\eeqn
We recall from \cite{liverani1999,aimino2015} that 
\begin{align*}
0 < \alpha \le \beta \hspace{0.2cm}  \Rightarrow \hspace{0.2cm} \cC_*(\alpha) \subset \cC_*(\beta),
\end{align*}
and that
\begin{align*}
0 < \alpha \le \beta\hspace{0.2cm}  \Rightarrow \hspace{0.2cm}  \cL_{\alpha} \cC_*(\beta) \subset  \cC_*(\beta).
\end{align*}

\subsection{Sequential compositions} First we consider sequential compositions 
\begin{align*}
\widetilde{T}_n = T_{\alpha_n} \circ \cdots \circ T_{\alpha_1}
\end{align*}
of intermittent maps with parameters $0 < \alpha_n \le \beta_* < 1$. The notation below is adapted from Section \ref{sec:self_rate}: $\mu$ is a Borel probability measure on $[0,1]$;  $g^n : [0,1] \to \bR^d$ is a bounded observable for all $n \ge 1$; 
\begin{align*}
W &= \sum_{i=0}^{N-1} b^{-1} \bar{f}^i; \hspace{0.5cm} b =  \left[ \operatorname{Cov}_\mu \left(\sum_{i=0}^{N-1} \bar f^i \right) \right]^{1/2}; \hspace{0.5cm} \bar{f}^n = g^n \circ \widetilde{T}_n - \mu(g^n \circ \widetilde{T}_n ); \\
&\lambda_{\min} = \text{the least eigenvalue of }  \operatorname{Cov}_\mu \left(\sum_{i=0}^{N-1} \bar f^i \right). \\
\end{align*}

For a Lipschitz continuous function $g : [0,1] \to \bR^d$ we set $\Vert g \Vert_{\text{Lip}} = \Vert g \Vert_{\infty} + \text{Lip}(g)$, where
\begin{align*}
\Vert g \Vert_\infty = \sup_{x \in [0,1]} \Vert g(x) \Vert
\end{align*}
and
\begin{align*}
\text{Lip}(g) = \sup_{x \neq y} \frac{ \Vert g(x) - g(y) \Vert }{|x-y|}. 
\end{align*}\medskip

\begin{thm}\label{thm:int_sequential} Let $N \ge 1$ and  let $\mu$ be a measure whose density lies in the cone $\cC_*(\beta_*)$. Suppose that $g^n : [0,1] \to \bR^d$ are Lipschitz continuous with $\sup_{n < N}\Vert g^n \Vert_{\textnormal{Lip}} + 1 \le L$ and that $\lambda_{\min} > 1$. Denote by $Z \sim \cN(0,I_{d\times d})$ a standard normal random vector.
\begin{itemize}
\item[(1)]{If $\beta_* < 1/3$, then there is $C_* = C_*( L, d, \beta_* ) > 0$ such that
\begin{align*}
d_{\mathscr{W}}(W,Z) \le C_* N  (1 + \log N) \lambda_{\min}^{-\frac32}.
\end{align*}
In particular, if $ \lambda_{\min} \gg N^{2/3} (\log N)^{2/3}$, then $W \stackrel{d}{\to}  Z$ as $N \to \infty$. } \smallskip

\item[(2)]{If $1/3 \le \beta_* < 2/5$, then for any $\delta > 0$ there is $C_* = C_*( L, d, \beta_*,\delta ) > 0$ such that 
\begin{align}\label{eq:large_param}
d_{\mathscr{W}}(W,Z) \le C_* N^{4 - \frac{1}{\beta_*} + \delta } \lambda_{\min}^{-\frac32}.
\end{align}
In particular,  if $ \lambda_{\min} \gg N^{8/3 + 2\delta/3 - 2/3\beta_*} $, then $W \stackrel{d}{\to}  Z$ as $N \to \infty$. 
}

\end{itemize}

\end{thm}

\begin{remark} A couple of remarks are in order:
\begin{itemize}
\item[(i)]{ The proof is based on Theorem \ref{thm:main_self}. In the special case $d = 1$ let us denote $S = \sum_{i=0}^{N-1} \bar{f}^i$ and $\sigma^2 = \mu(S^2)$. Assuming $\beta_* < 1/3$, the sharper upper bound
\begin{align*}
d_{\mathscr{W}}( \sigma^{-1}S,Z) \le C_* N \sigma^{-3}
\end{align*}
is obtained by applying Theorem \ref{thm:main_1d} instead of Theorem  \ref{thm:main_self}, provided that $\sigma^2 > 0$. Consequently, by Lemma \ref{lem:1d_change_normal}, for any $c > 0$,
\begin{align}\label{eq:int_1d_seq}
d_{\mathscr{W}}( c^{-1}S, c^{-1} \sigma Z) \le C_* N c^{-1} \sigma^{-2}.
\end{align}
Without any assumption on $\sigma^2$ we still obtain the weaker bound
\begin{align*}
d_{\mathscr{W}}( N^{-\frac12} S, N^{-\frac12}\sigma Z) \le C_* N^{-1/6}.
\end{align*}
This follows easily by combining  \eqref{eq:int_1d_seq} with the fact that, for any random variables $X$ and $Y$ with finite variances $\sigma_X^2$ and $\sigma_Y^2$, respectively, the Wasserstein metric satisfies $d_{\mathscr{W}}(X,Y) \le \sigma_X + \sigma_Y$ (see e.g. \cite{Hella_2018} for the last statement).} \smallskip
\item[(ii)]{In the stationary case of a single intermittent map $T_{\alpha}$ preserving the measure $\hat{\mu}_{\alpha}$, a Berry-Esseen theorem for univariate H\"{o}lder continuous observables was shown by Gou\"{e}zel \cite{gouezel2005}. Gou\"{e}zel's result establishes the rate $O(N^{-1/2})$ with respect to the Kolmogorov metric for parameters $\alpha < 1/3$. For parameters $1/3 \le \alpha < 1/2$, Gou\"{e}zel obtains a rate depending on the behavior of $f(x)$ around the fixed point $x = 0$. For multivariate Lipschitz continuous observables, the rate $O(N^{-1/2})$ in the CLT with respect to the Wasserstein metric was shown for parameters $\alpha < 1/3$ in \cite{Leppanen_2017} by an application of P{\`e}ne's theorem \cite{pene2005}. The upper bound \eqref{eq:large_param} can be viewed as an extension of this result for parameters $1/3 \le \alpha < 2/5$.  P{\`e}ne's condition (see  Section \ref{sec:pene}) does not hold for parameters $\alpha \ge 1/3$ because correlations do not decay at a rate which has a finite first moment.}
\end{itemize}
\end{remark}

\begin{proof}[Proof for Theorem \ref{thm:int_sequential}  ] Set $\rho(n) = n^{1-1/\beta_*} (\log n)^{1/\beta_*}$ for $n \ge 2$ and $\rho(0) = \rho(1) = 1$. We show that conditions (C1)-(C3) of Theorem \ref{thm:main_self} hold with $\rho$ using Theorem 1.1 in \cite{Leppanen_2017}.

(C1): Let $\alpha, \beta \in \{1,\ldots, d\}$ and $0 \le n,m \le N-1$. Applying Theorem 1.1 in \cite{Leppanen_2017} with $k = 2$, $p = 1$, $F(x,y,z) =g_{\alpha}^n (y)g_{\beta}^m(z)$, and $\mu_1 = \mu$ yields the upper bound
\begin{align*}
|\mu( \bar{f}_{\alpha}^n \bar{f}_{\beta}^m )| \le C L^2 \rho(|n-m|),
\end{align*}
where $C = C(\beta_*) > 0$.

(C2): Let $0 \le n,m \le N-1$, $m \le k \le N-1$, and $G : \bR^d \times B_d(0, 4L + 1   ) \to \bR^{d \times d}$ be a bounded $C^1$-function with bounded gradient. We define 
\begin{align*}
F(x_0, \ldots, x_{n-k}, x_{n-m}, x_n, x_{n+m}, x_{n+k}, \ldots,  x_{N-1} )
\end{align*}
by the formula
\begin{align*}
\psi^n(x_n)^T G\left(   \sum_{| i -n | > k } \psi^i(x_i),  \sum_{| i - n | = k} \psi^i(x_i)   \right) \sum_{| i - n | = m} \psi^i(x_i),
\end{align*}
where $\psi^i(x) = g^i(x) - \mu(g^i \circ \widetilde{T}_i)$ and the summations are over $i$. Then,
\begin{align*}
&\mu \left[  F(  \widetilde{T}_0, \ldots, \widetilde{T}_{n-k}, \widetilde{T}_{n-m}, \widetilde{T}_n, \widetilde{T}_{n+m}, \widetilde{T}_{n+k}, \ldots,  \widetilde{T}_{N-1}) \right] \\
&=\mu \left[  ( \bar{f}^{n})^T   G \left( \sum_{|i - n| > k} \bar{f}^i, \sum_{| i - n | = k} \bar{f}^i  \right)  \sum_{| i - n | = m} \bar{f}^i   \right],
\end{align*}
which is the integral we need to control. It is easy to verify that
\begin{align}\label{eq:pm_F_bound1}
\Vert F \Vert_{\infty} \le 8L   \Vert G \Vert_\infty
\end{align}
and 
\begin{align}\label{eq:pm_F_bound2}
\sup_{\beta \in \cI} \, [F]_{1,\beta} \le 8 L^3 d^2 ( \Vert G \Vert_{\infty} + \Vert \nabla G \Vert_\infty ).
\end{align}
Here
\begin{align*}
\Vert G  \Vert_\infty =  \sup_{x \in \bR^d, \, \Vert y \Vert < 4L + 1 } \Vert G(x,y) \Vert_s \hspace{0.3cm} \text{and}  \hspace{0.3cm} \Vert \nabla G \Vert_\infty =  \max_{1 \le i \le 2d} \sup_{x \in \bR^d, \, \Vert y \Vert < 4L + 1 } \Vert \partial_i G(x,y) \Vert_s,
\end{align*}
$[F]_{1,\beta}$ is defined by \eqref{eq:holder}, and $\cI = \{ 0 \le i \le N-1 \: : \: |i-n| \ge k \} \cup \{ 0 \le i \le N-1 \: : \: |i-n| = m \} \cup \{n\} $ is an indexing for the arguments of $F$. Theorem 1.1 in \cite{Leppanen_2017} together with \eqref{eq:pm_F_bound1} and  \eqref{eq:pm_F_bound2} implies the upper bound
\begin{align*}
&\left| \mu \left[  ( \bar{f}^{n})^T   G \left( \sum_{|i - n| > k} \bar{f}^i, \sum_{| i - n | = k} \bar{f}^i  \right)  \sum_{| i - n | = m} \bar{f}^i   \right]  \right|
\le C(\beta_*)  8 L^3 d^2 ( \Vert G \Vert_{\infty} + \Vert \nabla G \Vert_\infty )   \rho(m).
\end{align*}

(C3): This is shown in the same way as condition (C2). Namely Theorem 1.1 in \cite{Leppanen_2017} is applied with the function
\begin{align*}
\psi^n(x_n)^T G_1\left(   \sum_{| i -n | > k } \psi^i(x_i),  \sum_{| i - n | = k} \psi^i(x_i)   \right) \sum_{| i - n | = m} \psi^i(x_i),
\end{align*}
where
\begin{align*}
G_1(x,y) = G(x,y) - \mu \left[  G \left( \sum_{|i - n| > k} \bar{f}^i, \sum_{| i - n | = k} \bar{f}^i  \right)  \right].
\end{align*}
We leave the details to the reader.

If $\beta_* < 1/3$, it follows by the foregoing that conditions (C1)-(C3) hold also with $\rho(n) = n^{-\kappa}$ for some $\kappa > 2$. In particular $\sum_{m=1}^{\infty} ( 1 + \log( \rho(m)^{-1} ) ) m \rho(m) < \infty$, so that item (1) of Theorem \ref{thm:int_sequential} follows by Theorem \ref{thm:main_self}. If instead $1/3 \le \beta_*  < 2/5$ we obtain conditions (C1)-(C3) with $\rho(n) = n^{1 - 1/\beta_* + \delta}$ for any $\delta > 0$. Then $\sum_{m=1}^{N-1}  (1 + \log (\rho(m)^{-1} ) ) m\rho(m) \le C(\beta_*,\delta, \delta' ) N^{3 - 1/\beta_* + \delta + \delta' }$ holds for arbitrarily small $\delta' > 0$ and item (2) of Theorem  \ref{thm:int_sequential}  follows again by Theorem \ref{thm:main_self}.
\end{proof}

In the remainder of this section we look at situations where we have control on the limiting behavior of $\operatorname{Cov}_\mu (\sum_{i=0}^{N-1} \bar f^i )$.

\subsection{Quasistatic dynamics} We apply Theorem \ref{thm:int_sequential} to a model described by time-dependent (non-random) compositions of slowly transforming intermittent maps. More precisely we consider the following subclass of quasistatic dynamical systems (QDS); for background and earlier results on quasistatic systems we refer the reader to \cite{Leppanen_2018,Stenlund_2016,Leppanen_2016, dobbs2017, Hella_2018, HL_2018}.

\begin{defn}[Intermittent QDS]\label{def_int_qds} Let $\bfT = \{ T_{\alpha_{n,k}} \: : \: 0\le k\le n, \ n\ge 1 \}$ be a triangular array of intermittent maps with parameters $\alpha_{n,k} \in [0,1)$. If there is a piecewise continuous curve $\gamma : [0,1] \to [0,1)$ satisfying
\beqn
\lim_{n\to\infty} \alpha_{n,\round{nt}} = \gamma_t
\eeqn
for all $t$, we say that $(\bfT, \gamma)$ is an intermittent QDS.
\end{defn}

Given an intermittent QDS $(\bfT, \gamma)$, we define the functions $S_n:[0,1]\times [0,1]\to\bR$ by
\beqn
S_n(x,t) =   \int_0^{nt} f_{n,\round{s}}(x)\, ds, \quad n\ge 1,
\eeqn
where 
\beqn
f_{n,k} = f\circ T_{n,k}\circ\dots\circ T_{n,1}, \quad 0\le k \le n,
\eeqn
$f_{n,0} = f$, and $f : [0,1] : \to \bR^d$ is a bounded function. We fix an initial distribution $\mu$ of $x \in [0,1]$ and for each $t \in [0,1]$ view the $S_n(t) = S_n(\cdot,t)$ as random vectors. The problem is now to approximate the law of the fluctuations
\begin{align*}
W_n(x,t) = b^{-1} \bar{S}_n(x,t)
\end{align*}
by $\cN(0,I_{d \times d})$, where $\bar{S}_n(x,t) = S_n(x,t) - \mu(S_n(x,t))$ and $b = b(n,t) = \text{Cov}_\mu(\bar{S}_n(\cdot,t))^{1/2}$.

\begin{thm}\label{thm:qds_pm} Let \(f : [0,1] \to \bR^d\) be a Lipschitz continuous function and \(\mu\) be such that its density lies in \(\cC_*(\beta_*)\). Suppose that the limiting curve $\gamma$ is Hölder-continuous, that for some $\eta \in (0,1]$ we have 
\begin{align*}
\sup_{n\ge 1} n^{\eta}\sup_{t \in [0,1]} | \alpha_{n,\round{nt}} - \gamma_t | < \infty,
\end{align*}
and that there exists \(t_{0}\in (0,1]\) such that $f$ is not a co-boundary for $T_{\gamma_{t_0}}$ in any direction\footnote{i.e.  there does not exist a unit vector $v \in \bR^d$, a constant $c \in \bR^d$, and a function $\psi \in L^2(\hat{\mu}_{\gamma_{t_0}} )$ such that $v^T f = c + \psi - \psi \circ T_{\gamma_{t_0}}$.}. 

\begin{itemize}
\item[(1)]{If \(\gamma([0,1]) \subset [0,\beta_*]\) and \(\beta_* < 1/3\), then there exists $C_* = C_*(t_0, d, f , \gamma )$ such that for all $t \ge t_0$ and $n \ge 2$,
\begin{align*}
d_{\mathscr{W}}(W_n(t),Z) \le C_* n^{-\frac12}  \log n.
\end{align*}} \smallskip
\item[(2)]{If \(\gamma([0,1]) \subset [0,\beta_*]\) and \(1/3 \le \beta_* < 2/5\), then for any $\delta > 0$ there exists $C_* = C_*(t_0, d, f , \delta, \gamma)$ such that for all $t \ge t_0$ and $n \ge 1$,
\begin{align*}
d_{\mathscr{W}}(W_n(t),Z) \le C_* n^{\frac52 - \frac{1}{\beta_*} + \delta }.
\end{align*}}

\end{itemize}

 \end{thm}
 
\begin{proof} Set $\xi_n(x,t) = n^{-\frac12} b W_n(x,t) $. By Lemma 4.4 in \cite{HL_2018}, uniformly in $t \in [0,1]$, 
\begin{align*}
\lim_{n \to \infty} [ \text{Cov}_\mu( \xi_n(t) )]_{\alpha,\beta} = \int_0^t [\hat{\Sigma}_s(f)]_{\alpha,\beta}  \, ds \hspace{0.5cm} \forall \alpha, \beta \in \{1,\ldots, d\},
\end{align*}
where
\begin{align*}
\hat{\Sigma}_t(f) = \hat{\mu}_{\gamma_t}[\hat{f}_t \otimes \hat{f}_t ] + \sum_{k=1}^{\infty} (\hat{\mu}_{\gamma_t}[\hat{f}_t \otimes \hat{f}_t \circ T_{\gamma_t}^k] + \hat{\mu}_{\gamma_t}[\hat{f}_t \circ T_{\gamma_t}^k \otimes \hat{f}_t ]),
\end{align*}
and $\hat{f}_t = f - \hat{\mu}_{\gamma_t}(f)$. By theorem 2.11 in the same article the limit covariance $\Sigma_t(f) : = \int_0^t \hat{\Sigma}_s(f)  \, ds$ is positive definite for all $t \ge t_0$ (this is where the co-boundary condition on $f$ is needed). In particular, $\lambda_{\min} (\Sigma_t(f)) > 0$, where $\lambda_{\min}(A)$ denotes the least eigenvalue of the matrix $A \in \bR^{d \times d}$. It follows by the same argument as in p. 20 of \cite{HL_2018} that there exists $n_0$ and $C > 0$ such that $\lambda_{\min}(  \text{Cov}_\mu( \xi_n(t) )) \ge C  $ holds for all $t \ge t_0$ and all $n \ge n_0$. In other words, 
\begin{align}\label{eq:lb_eigenv}
\lambda_{\min} ( \text{Cov}_\mu( \bar{S}_n(t))) \ge C n \hspace{0.5cm} \forall t \ge t_0 \, \forall n \ge n_0.
\end{align}

Next we show the wanted upper bound on $d_{\mathscr{W}}( W_n(t), Z) $ by controlling separately the following three terms:
\begin{align}\label{eq:term_1}
d_{\mathscr{W}}\left( b(   \tfrac{ \lceil nt\rceil }{ n}    )^{-1} \bar{S}_n(  \tfrac{ \lceil nt\rceil }{ n}  ), Z \right),
\end{align}
\begin{align}\label{eq:term_2}
d_{\mathscr{W}}\left( b(   \tfrac{ \lceil nt\rceil }{ n}    )^{-1} \bar{S}_n(  \tfrac{ \lceil nt\rceil }{ n}  ), b(  t  )^{-1} \bar{S}_n(  \tfrac{ \lceil nt\rceil }{ n}  ) \right),
\end{align}
\begin{align}\label{eq:term_3}
d_{\mathscr{W}}\left( b(  t  )^{-1} \bar{S}_n(  \tfrac{ \lceil nt\rceil }{ n}  ), b(  t  )^{-1} \bar{S}_n( t) \right),
\end{align}
where  $b(s) = \text{Cov}_\mu(\bar{S}_n(s))^{1/2}$.

Note that
\begin{align*}
\bar{S}_n(  \tfrac{ \lceil nt\rceil }{ n}  ) = \sum_{k=0}^{ \lceil nt\rceil -1 } \bar{f}_{n,k}.
\end{align*}
It follows immediately by \eqref{eq:lb_eigenv} and Theorem \ref{thm:int_sequential} that for all $n \ge n_0$ and $t \ge t_0$,
\begin{align*}
\eqref{eq:term_1}  \le  \begin{cases}  C( \Vert f \Vert_{\textnormal{Lip}}, d, \beta_*, t_0) n^{-\frac12}   \log n, \hspace{0.2cm}  &0 \le \beta_* < 1/3 \\
C( \Vert f \Vert_{\textnormal{Lip}}, d, \beta_*, t_0, \delta) n^{\frac52 - \frac{1}{\beta_*} + \delta },  \hspace{0.2cm} &1/3 \le \beta_* < 2/5,
\end{cases}
\end{align*}
where $\delta > 0$ can be made arbitrarily small.

In the remainder of this proof we assume that $\beta_* < 1/2$ and $\gamma([0,1]) \subset [0,\beta_*]$. Whenever $t \ge t_0$ and $n \ge n_0$,
\begin{align*}
\eqref{eq:term_2}  &\le  \Vert b(   \tfrac{ \lceil nt\rceil }{ n}    )^{-1}  -   b(  t  )^{-1}  \Vert_s \mu \left( \Vert \bar{S}_n(  \tfrac{ \lceil nt\rceil }{ n}  )   \Vert \right) \\
&\le \Vert  b(   \tfrac{ \lceil nt\rceil }{ n}    )^{-1}   \Vert_s  \Vert b(   \tfrac{ \lceil nt\rceil }{ n}    ) -   b(  t  )  \Vert_s  \Vert  b(  t  )^{-1}   \Vert_s \mu \left( \Vert \bar{S}_n(  \tfrac{ \lceil nt\rceil }{ n}  )   \Vert^2 \right)^{\frac12}.
\end{align*}
Since the density of $\mu$ belongs to $\cC_*(\beta_*)$, it follows by Lemma 3.3 in \cite{Leppanen_2018} that
\begin{align*}
\mu \left( \Vert \bar{S}_n(  \tfrac{ \lceil nt\rceil }{ n}  )   \Vert^2 \right) \le Cn.
\end{align*}
where $C = C(d, \Vert f \Vert_{\text{Lip}}, \gamma, \beta_*) > 0$ is a constant independent of $t$. Moreover (see Lemma \ref{lem:spectral}),
\begin{align*}
\Vert  b(   \tfrac{ \lceil nt\rceil }{ n}    )^{-1}   \Vert_s = \lambda_{\min} \left( \text{Cov}_\mu \left(\bar{S}_n(\tfrac{ \lceil nt\rceil }{ n}  ) \right)  \right)^{-\frac12} \le C n^{-\frac12}
\end{align*}
and
\begin{align*}
\Vert  b(   t    )^{-1}   \Vert_s = \lambda_{\min} \left( \text{Cov}_\mu \left( \bar{S}_n(t ) \right)  \right)^{-\frac12} \le C n^{-\frac12}.
\end{align*}
That is,
\begin{align*}
\eqref{eq:term_2} \le Cn^{-\frac12}  \Vert b(   \tfrac{ \lceil nt\rceil }{ n}    ) -   b(  t  )  \Vert_s.
\end{align*}
For brevity denote $\Sigma_{n,s} = \text{Cov}_\mu (\bar{S}_n(s))$. Then we have for $t \ge t_0$ and $n \ge n_0$ the upper bound (see \cite{Schmitt_1992} for the first inequality)
\begin{align*}
\Vert b(   \tfrac{ \lceil nt\rceil }{ n}    ) -   b(  t  )  \Vert_s \le  \frac{\Vert \Sigma_{n,\lceil nt\rceil / n  }  - \Sigma_{n,t} \Vert_s }{  \lambda_{\min}(\Sigma_{n,\lceil nt\rceil / n  } )^{\frac12} + \lambda_{\min}( \Sigma_{n,t})^{\frac12} } \le C n^{-\frac12} \Vert \Sigma_{n,\lceil nt\rceil / n  }  - \Sigma_{n,t} \Vert_s.
\end{align*}
To bound the remaining spectral norm we fix $\alpha, \beta \in \{1,\ldots, d\}$, denote $\varphi = f_{\alpha}$ and $\psi = f_{\beta}$. For a a real-valued function $g : [0,1] \to \bR$ and integers $0 \le k \le n$ we denote $\bar{g}_{n, k} = g \circ T_{n,k} - \mu(g \circ T_{n,k})$. Whenever $n \ge 2/t_0$, we can use Theorem 1.1 in \cite{Leppanen_2018} to find $\kappa > 1$ such that 
\begin{align*}
 &\left|  [\Sigma_{n,\lceil nt\rceil / n  }  - \Sigma_{n,t}]_{\alpha,\beta} \right| \\
 &= \left| \int_0^{\lceil nt \rceil} \int_0^{\lceil nt \rceil} \mu( \bar{\varphi}_{n,\round{r}} \bar{\psi}_{n,\round{s}} )  \, dr \, ds  - \int_0^{nt} \int_0^{nt} \mu( \bar{\varphi}_{n,\round{r}} \bar{\psi}_{n,\round{s}} )  \, dr \, ds \right| \\
 &\le \int_0^{\lceil nt \rceil} \int_{nt}^{\lceil nt \rceil} | \mu( \bar{\varphi}_{n,\round{r}} \bar{\psi}_{n,\round{s}} ) |  \, dr \, ds  +  \int_{nt}^{\lceil nt \rceil} \int_{0}^{\lceil nt \rceil} | \mu( \bar{\varphi}_{n,\round{r}} \bar{\psi}_{n,\round{s}} ) |  \, dr \, ds   \\
  &\le C\Vert f \Vert_{\infty}^2 +  \int_0^{nt-2} \int_{nt}^{\lceil nt \rceil} | \mu( \bar{\varphi}_{n,\round{r}} \bar{\psi}_{n,\round{s}} ) |  \, dr \, ds  
  +  \int_{nt}^{\lceil nt \rceil} \int_{0}^{ nt -2 } | \mu( \bar{\varphi}_{n,\round{r}} \bar{\psi}_{n,\round{s}} ) |  \, dr \, ds    \\
  &\le C \Vert f \Vert_{\infty}^2 + 2C \Vert f \Vert_{\text{Lip}}^2  \int_0^{nt-2} \int_{nt}^{\lceil nt \rceil} ( \round{r}- \round{s})^{-\kappa} \, dr \, ds \\
  &\le C \Vert f \Vert_{\text{Lip}}^2.
\end{align*}
Hence, the upper bound $\Vert \Sigma_{n,\lceil nt\rceil / n  }  - \Sigma_{n,t} \Vert_s \le dC \Vert f \Vert_{\text{Lip}}^2$ follows by Lemma  \ref{lem:spectral}. We have shown that $\eqref{eq:term_2} \le C n^{-1/2}$ whenever $t \ge t_0$ and $n \ge n_0$.

Finally, by \eqref{eq:lb_eigenv} and Lemma \ref{lem:spectral},
\begin{align*}
\eqref{eq:term_3} &\le \mu \left[ \Vert b(t)^{-1}(\bar{S}_n( \tfrac{  \lceil nt\rceil}{n}) -  \bar{S}_n(t)  )  \Vert  \right] \le \Vert b(t)^{-1} \Vert_s  \int^{\lceil nt \rceil}_{nt}  \mu (\Vert \bar{f}_{n,\round{r}} \Vert ) \, dr \\
&\le C \Vert f \Vert_{\infty} n^{-\frac12},
\end{align*}
whenever $t \ge t_0$ and $n \ge n_0$. Now to finish the proof for Theorem \ref{thm:qds_pm} it suffices to combine the foregoing upper bounds on $\eqref{eq:term_1}$, \eqref{eq:term_2}, and \eqref{eq:term_3}.

\end{proof}

\subsection{Rate in the quenched CLT} We consider a sequence $(T_{\omega_{i}})_{i\ge1}$ of intermittent maps with parameters $(\omega_i)_{i \ge 1}$ drawn randomly from the probability space~$(\Omega,\cF,\bP) = ([0,\beta_*]^{\bZ_+},\cE^{\bZ_+},\bP)$, where $\cE$ is the Borel $\sigma$-algebra of $[0,\beta_*]$ and $\bZ_+ = \{1,2,\dots\}$. Let $\tau:\Omega\to\Omega$ denote the shift $(\tau(\omega))_i = \omega_{i+1}$.

\noindent \textbf{Conditions (RDS):}
\begin{itemize}
\item[(i)]{The shift $\tau:\Omega\to\Omega:(\tau(\omega))_i = \omega_{i+1}$ preserves $\bP$.} \smallskip
\item[(ii)]{There is $C> 0$ and $\gamma > 0$ such that, for all $n \ge 1$,
\beqn
\sup_{i\ge 1} \sup_{A\in \cF_1^i,\, B\in \cF_{i+n}^\infty}|\bP(A \cap B) - \bP(A)\,\bP(B)| \le Cn^{-\gamma},
\eeqn
where $\cF_1^i$ is the sigma-algebra generated by the projections $\pi_1,...,\pi_i$, $\pi_k(\omega)=\omega_k$, and $\cF_{i+n}^\infty$ is generated by $\pi_{i+n},\pi_{i+n+1}, \dots$.}
\end{itemize}
We set 
\begin{align*}
W = W(\omega) = \sum_{n=1}^{N-1} N^{-\frac12} \bar{f}^n ; \hspace{0.5cm}  \bar{f}^n =  f \circ \varphi(n,\omega) - \mu(f \circ \varphi(n,\omega));
\end{align*}
where $f : [0,1] \to \bR^d$ is a bounded observable with $d \ge 1$, and $\varphi(n,\omega) = T_{\omega_n} \circ \cdots \circ T_{\omega_1}$. That is, we take
\begin{align*}
b = \sqrt{N} I_{d \times d}
\end{align*}
as the normalizing matrix.

\begin{thm}\label{thm:int_random} Suppose that $\beta_* < 1/3$, that $f : [0,1] \to \bR^d$ is Lipschitz continuous, and that $\mu$ is a measure whose density belongs to $\cC_*(\beta_*)$. Assume conditions (RDS). Then, 
 \beqn
 \Sigma= \sum_{k=0}^\infty (2-\delta_{k0})\lim_{i\to\infty}  \bE [\mu(f^i (f^{i+k})^T) - \mu(f^i)\mu( (f^{i+k})^T)]
 \eeqn
is well-defined and positive semi-definite. Moreover, $\Sigma$ is positive definite if and only if
\begin{align}\label{eq:covar_pos}
\sup_{N \ge 1} \bE  \mu \left( \sum_{i < N} v^T f \circ \varphi(i,\omega)  \right)^2 = \infty
\end{align}
holds for all $v \neq 0$. Fix an arbitrarily small $\delta > 0$. If $\Sigma$ is positive definite, then there is $\Omega^* \subset \Omega$ with $\bP(\Omega^*) = 1$ such that for any three times differentiable function  \(h: \, \bR^d \to \bR\) with \(\max_{1 \le k \le 3} \Vert D^k h \Vert_{\infty} < \infty\), any $\omega \in \Omega^*$, and any $N \ge 2$, 
\begin{align*}
|\mu[h(W)] - \Phi_{\Sigma}(h)| \le C_*(\Vert D^3 h \Vert_{\infty} + \Vert h \Vert_{\textnormal{Lip}})  \theta(N),
\end{align*}
where $C_* = C_*( (T_{\omega})_{\omega \in \Omega}, d, f ) > 0$ and
\begin{align*}
\theta(N) = 
\begin{cases}
 O(N^{-\frac 12} (\log N)^{\frac32+\delta}), & \gamma > 1,
 \\
  O(N^{-\frac12+\delta}), & \gamma = 1,
  \\
  O(N^{-\frac{\gamma}{2}} (\log N)^{\frac32+\delta}), & 0<\gamma<1.
\end{cases}
\end{align*}
\end{thm}

\begin{remark}  Nicol--T\"{o}r\"{o}k--Vaienti  \cite{Nicol_2018}, Su \cite{Su_vector_2019}, and Nicol--Pereira--T\"{o}r\"{o}k \cite{nicol2019} proved CLTs without rates of convergence for random dynamical systems of intermittent maps with parameters $\omega_i \le \beta_* < 1/2$. Theorem \ref{thm:int_random} gives a better rate of convergence than the following upper bound established in \cite{HL_2018} for univariate $f : [0, 1] \to \bR$:
 \beqn
d_\mathscr{W}(W,\sigma Z) = 
\begin{cases}
O(N^{\beta_* - \frac12}(\log N)^{\frac{1}{\beta_*}}), & \gamma \ge 1,
  \\
  O(N^{\beta_* - \frac12}(\log N)^{\frac{1}{\beta_*}})+O(N^{-\frac{\gamma}{2}} (\log N)^{\frac32+\delta}), &  0<\gamma<1.
\end{cases}
 \eeqn 
Here $Z \sim \cN(0, 1)$ and $\sigma^2 = \Sigma$. The proof below can be modified to obtain the upper bound $d_\mathscr{W}(W,\sigma Z) \le C_* \theta(N)$ for univariate $f : [0,1] \to \bR$.

\end{remark}

\begin{proof} Given any vector $v \in \bR^d$ denote
\begin{align*}
\ell_n(v) = v^T \Sigma_N v,
\end{align*}
where $\Sigma_N = \text{Cov}_\mu(W \otimes W)$. In other words, $\ell_n(v)$ is the variance of 
\begin{align*}
W(v^Tf) = \sum_{n=1}^{N-1} N^{-\frac12} v^T \bar{f}^n.
\end{align*}
Let $v \in \bR^d$. By Theorem 2.6 in \cite{HL_2018}, $\lim_{n \to \infty} \ell_n(v) = v^T \Sigma v$ exists and $v^T \Sigma v > 0$ if and only if
\begin{align*}
\sup_{N \ge 1} \bE  \mu \left( \sum_{i < N} v^T f \circ \varphi(i,\omega)  \right)^2 = \infty.
\end{align*}
Hence \eqref{eq:covar_pos} is equivalent to the positive definiteness of $\Sigma$. The proof of Theorem 2.6 in \cite{HL_2018} also shows that, for almost every $\omega \in \Omega$, 
\beq
\left|\ell_{N}(v)- v^T \Sigma v \right| = 
\begin{cases}
 O(N^{-\frac 12} (\log N)^{\frac32+\delta}), & \gamma > 1,
 \\
  O(N^{-\frac12+\delta}), & \gamma = 1,
  \\
  O(N^{-\frac{\gamma}{2}} (\log N)^{\frac32+\delta}), & 0<\gamma<1,
  \label{eq:sigma^2 bounds}
\end{cases}
\eeq
Hence, by Lemma 4.4 in \cite{Stenlund_2018}, for almost every $\omega \in \Omega$, 
\begin{align}
\Vert \Sigma_N - \Sigma \Vert_s = 
\begin{cases}
 O(N^{-\frac 12} (\log N)^{\frac32+\delta}), & \gamma > 1,
 \\
  O(N^{-\frac12+\delta}), & \gamma = 1,
  \\
  O(N^{-\frac{\gamma}{2}} (\log N)^{\frac32+\delta}), & 0<\gamma<1.
\end{cases}
 \label{eq:sigma^2_bounds}
\end{align}

From now on we assume that $\Sigma$ is positive definite. We split $|\mu[h(W)] - \Phi_{\Sigma}(h)| $ into two terms:
\begin{align}\label{eq:random_pm_term1}
|\mu[h(W)] - \Phi_{\Sigma_N}(h)| 
\end{align}
and
\begin{align}\label{eq:random_pm_term2}
| \Phi_{\Sigma_N}(h) - \Phi_{\Sigma}(h)|.
\end{align}
It follows by \eqref{eq:sigma^2_bounds} that there is $N_0$ such that $\Sigma_N$ is positive definite for $N \ge N_0$ and a.e. $\omega \in \Omega$. Then, for all $\omega \in \Omega$ and $N \ge N_0$, the upper bound
\begin{align}\label{eq:pm_random_normal}
\eqref{eq:random_pm_term1} \le C  \Vert D^3 h \Vert_{\infty} N^{-\frac12}
\end{align}
holds for some $C = C(\beta_*, d, \Vert f \Vert_{\text{Lip}}) > 0$. The proof for \eqref{eq:pm_random_normal} is almost verbatim the same as the proof for Theorem \ref{thm:int_sequential}: Theorem \ref{thm:main_multi_general} is applied with $b = \sqrt{N} I_{d \times d}$ after verifying conditions (A1)-(A3) using Theorem 1.1 in \cite{Leppanen_2018}. We will not repeat the argument here.

Finally, it is easy to show that, for some absolute constant $C > 0$,
\begin{align*}
\eqref{eq:random_pm_term2} \le C \text{Lip}(h) \Vert \Sigma_N^{\frac12} - \Sigma^{\frac12} \Vert_s.
\end{align*}
Hence, for $N \ge N_0$ and a.e. $\omega \in \Omega$ (see \cite{Schmitt_1992}  for the first inequality),
\begin{align*}
\eqref{eq:random_pm_term2} \le  C \text{Lip}(h)\frac{\Vert \Sigma_{N}  - \Sigma \Vert_s }{  \lambda_{\min}(\Sigma_{N} )^{\frac12} + \lambda_{\min}( \Sigma)^{\frac12} } \le C \text{Lip}(h) \lambda_{\min}( \Sigma)^{-\frac12}  \Vert \Sigma_{N}  - \Sigma \Vert_s.
\end{align*}
The obtained upper bound combined with \eqref{eq:sigma^2_bounds} finishes the proof for Theorem \ref{thm:int_random}.

\end{proof}

\section{Proofs for main results}\label{sec:proof}

\subsection{On the regularity of solutions to Stein equation} Let the matrix~\(\Sigma \in \bR^{d\times d}\) be symmetric and positive definite. Denote respectively by~$\phi_\Sigma$ and~$\Phi_\Sigma$ the density and expected value of the $d$-dimensional normal distribution with mean $0$ and covariance matrix $\Sigma$. Given a test function \(h:\bR^d\to\bR\), define
\begin{align}\label{eq:stein_solution}
&A(w)= -\int_0^{\infty} \! \left\{\int_{\bR^d} h(e^{-s}w + \sqrt{1 - e^{-2s}}\,z)\,\phi_{\Sigma}(z)\,dz - \Phi_{\Sigma}(h)\right\} \, ds.
\end{align}
Then, we have the following result for smooth test functions $h$; see~\cite{Barbour_1990,Gotze_1991,GoldsteinRinott_1996,Gaunt_2016}:

\begin{lem}\label{lem:stein_solution} Let \(h: \, \bR^d \to \bR\) be three times differentiable with \(\Vert D^k h \Vert_{\infty} < \infty\) for \(1 \le k \le 3\). Then,  \(A \in C^3(\bR^d,\bR)\), and \(A\) solves the Stein equation \eqref{eq:stein_multi}. Moreover, the partial derivatives of \(A\) satisfy the bounds 
 \begin{align*}
\| \partial_1^{t_1}\cdots \partial_d^{t_d} A \|_\infty \le k^{-1} \| \partial_1^{t_1}\cdots \partial_d^{t_d} h \|_\infty
\end{align*}
whenever $t_1+\cdots+t_d = k$, $1\le k\le 3$.
\end{lem}

Note that the bounds on the partial derivatives of $A$ are independent of the covariance matrix $\Sigma$.

Recently Gallou\"et--Mijoule--Swan \cite{Gallouet_2018} obtained notable improvements on the regularity of solutions to Stein's equation in the case $\Sigma = I_{d\times d}$, for test functions $h$ that are assumed to be Hölder continuous:

\begin{lem}[See Proposition 2.2 in \cite{Gallouet_2018} ]\label{lem:solution_self} Set $\Sigma = I_{d\times d}$ and let $h : \bR^d \to \bR$ be $\eta$-Hölder continuous with some $\eta \in (0,1]$. Then the function $A : \bR^d \to \bR$ defined by \eqref{eq:stein_solution} solves the Stein equation \eqref{eq:stein_multi}. Moreover, $A \in C^2(\bR^d,\bR)$  and its second derivative satisfies the following bound:
\begin{align}\label{eq:log_lip}
\Vert D^2A(w) - D^2A(w')  \Vert_s &\le \Vert w - w' \Vert [h]_\eta (  C_\# +  |\log  \Vert w - w' \Vert | ), \hspace{0.2cm} \forall w,w' \in \bR^d,
\end{align}
where
\begin{align*}
[h]_\eta = \sup_{x \neq y} \frac{| h(x) - h(y) |}{ \Vert x - y \Vert^{\eta}}
\end{align*}
and
\begin{align*}
C_\#  = 2^{\frac{\eta}{2} + 1 }  \frac{\eta + 2d}{\eta d} \frac{\Gamma( \frac{\eta + d}{2})}{\Gamma(\frac{d}{2})}.
\end{align*}
\end{lem}

We will apply the result with $\eta = 1$. In this case the result is known to be optimal in terms of regularity of $D^2A$. More precisely it was shown in  \cite{Gallouet_2018}  that, when $d=2$ and $h(x,y) = \max \{0 , \min \{x,y\} \}$ (an example considered first by Rai\v{c} in \cite{Raic_2004}),
\begin{align*}
\partial_i \partial_j A(u,u) - \partial_i \partial_j A(0,0) \sim_{u \to 0+} \frac{1}{\sqrt{2\pi}} u \log u.
\end{align*}

\subsection{Sunklodas' decomposition.} Set $Y^i = b^{-1} \bar{f}^i$ so that 
\begin{align*}
W = \sum_{i=0}^{N-1} Y^i.
\end{align*}

Next, we define punctured modifications of the sum $W$, namely
\beqn
W^{n,m} = W - \sum_{i\in[n]_m} Y^i,
\eeqn
where
\beqn
[n]_m = \{0\le i < N : |i-n|\le m\}.
\eeqn
Moreover, set
\begin{align*}
Y^{n,m} = b^{-1} \bar{f}^{n,m} = \sum_{|i - n | = m} b^{-1} \bar{f}^i.
\end{align*}
Note that 
\beq\label{eq:W_cases_vector}
W = \sum_{k=0}^{N-1} Y^{n,k} = W^{n,-1} 
\quad\text{and}\quad
W^{n,N-1} = 0
\eeq
as well as
\beq\label{eq:Wnm_rec_vector}
W^{n,m-1} = W^{n,m} + Y^{n,m}
\eeq
for any $n$ and $m$.

The proofs for the main results are based on the following decomposition, which is a multivariate version of Proposition 4 in \cite{sunklodas2007} due to Sunklodas.

\begin{prop}\label{prop:sunklodas_multi}
Suppose $A \in C^2(\bR^d, \bR)$. Denote
\beqn
\delta^{n,m}(u) = D^2A(W^{n,m} + u\, Y^{n,m}) - D^2A(W^{n,m})
\eeqn
and
\beqn
\delta^{n,m} = \delta^{n,m}(1) = D^2 A(W^{n,m-1}) - D^2A(W^{n,m}).
\eeqn
Then
\beqn
\mu[ \textnormal{tr} \Sigma D^2 A(W) - W^T \nabla A(W)] = E_1 + \cdots + E_7,
\eeqn
where
\begin{align*}
E_1 & = - \sum_{n=0}^{N-1}\sum_{m=1}^{N-1} \int_0^1  \, \mu[  (Y^{n})^T \delta^{n,m}(u) Y^{n,m}   ] \, du,
\\
E_2 & = - \sum_{n=0}^{N-1} \int_0^1   \mu[ (Y^{n})^T \delta^{n,0}(u) Y^n ] \, du,
\\
E_3 & = - \sum_{n=0}^{N-1}\sum_{m=1}^{N-1} \sum_{k=m+1}^{2m}    \mu [  (Y^{n})^T \, \overline{\delta^{n,k}} \, Y^{n,m}  ],
\\
E_4 & = - \sum_{n=0}^{N-1}\sum_{m=1}^{N-1} \sum_{k=2m+1}^{N-1}     \mu [  (Y^{n})^T \, \overline{\delta^{n,k}} \, Y^{n,m}  ],
\\
E_5 & = - \sum_{n=0}^{N-1} \sum_{k=1}^{N-1}  \mu [  (Y^{n})^T \, \overline{\delta^{n,k}} \, Y^n  ],
\\
E_6 & = \sum_{n=0}^{N-1}\sum_{m=1}^{N-1}  \mu \left[  (Y^{n})^T   \sum_{k=0}^m \mu( \delta^{n,k} )  Y^{n,m}  \right],
\\
E_7 & =  \sum_{n=0}^{N-1}  \mu [  (Y^{n})^T   \mu( \delta^{n,0} )  Y^n  ].
\end{align*}

\end{prop}

\begin{proof} For any $n$, by~\eqref{eq:W_cases_vector},
\beqn
\nabla A(W) - \nabla A(0) = \nabla A(W^{n,-1}) - \nabla A(W^{n,N-1}) = \sum_{m=0}^{N-1} [\nabla A(W^{n,m-1}) - \nabla A(W^{n,m})].
\eeqn
By~\eqref{eq:Wnm_rec_vector},
\begin{align*}
&\nabla A(W^{n,m-1}) - \nabla A(W^{n,m}) \\
&= \!\left[ D^2 A(W^{n,m}) + \int_0^1 D^2 A(W^{n,m}+u\, Y^{n,m}) - D^2 A(W^{n,m})\,\rd u \right] Y^{n,m} \\
&= \!\left[ D^2 A(W^{n,m}) + \int_0^1  \delta^{n,m}(u) \,\rd u \right] Y^{n,m} .
\end{align*}
Since $\mu[ (Y^n)^T   \nabla A(0)] = 0$, it follows by the above identities that 
\begin{align*}
\mu[ W^T \nabla A(W) ] &= \sum_{n=0}^{N-1} \mu[ (Y^n)^T (\nabla A(W)- \nabla A(0))] \\
&= -E_1 - E_2  + \sum_{n=0}^{N-1}\sum_{m=0}^{N-1} \mu[ (Y^{n})^T  D^2A (W^{n,m}) Y^{n,m} ].
\end{align*}
Note that
\begin{align*}
\mu[ \textnormal{tr} \Sigma D^2 A(W) ] = \mu[W^T \mu(D^2A(W)) W]
 = \sum_{n=0}^{N-1}\sum_{m=0}^{N-1} \mu [ (Y^{n})^T \mu (D^2A(W)) Y^{n,m}],
\end{align*}
so what remains of $\mu[ \textnormal{tr} \Sigma D^2 A(W) - W^T \nabla A(W)]$ after subtracting $E_1$ and $E_2$ is
\beqn
\begin{split}
& \sum_{n=0}^{N-1}\sum_{m=0}^{N-1} \mu [ (Y^{n})^T \mu (D^2A(W)) Y^{n,m}  - (Y^{n})^T  D^2A(W^{n,m}) Y^{n,m} ]  
\\
=\ & \sum_{n=0}^{N-1}\sum_{m=0}^{N-1} \!\left( \mu [ (Y^{n})^T \mu (D^2A(W) - D^2A(W^{n,m})) Y^{n,m} ] - \mu [ (Y^{n})^T  \overline{D^2A(W^{n,m})} Y^{n,m} ]  \right),
\end{split}
\eeqn
where
\beqn
\overline{D^2A(W^{n,m})} = D^2A(W^{n,m}) - \mu[D^2A(W^{n,m}) ].
\eeqn
Next note that 
\beqn
D^2 A(W^{n,m}) - D^2A(0) = D^2A(W^{n,m}) - D^2A(W^{n,N-1}) = \sum_{k=m+1}^{N-1} \delta^{n,k}.
\eeqn
Since $\mu[  (Y^{n})^T \overline{D^2A(0)} Y^{n,m}  ] = 0$, this yields
\beqn
\begin{split}
& \sum_{n=0}^{N-1}\sum_{m=0}^{N-1} \mu [ (Y^{n})^T  \overline{D^2A(W^{n,m})} Y^{n,m} ] =  \sum_{n=0}^{N-1}\sum_{m=0}^{N-1} \sum_{k=m+1}^{N-1}  \mu [ (Y^{n})^T  \overline{\delta^{n,k} } Y^{n,m} ]
\\
&= \sum_{n=0}^{N-1}\sum_{m=1}^{N-1} \sum_{k=m+1}^{2m}   \mu [ (Y^{n})^T  \overline{\delta^{n,k} } Y^{n,m} ] + \sum_{n=0}^{N-1}\sum_{m=1}^{N-1} \sum_{k=2m+1}^{N-1}   \mu [ (Y^{n})^T  \overline{\delta^{n,k} } Y^{n,m} ]\\
&+ \sum_{n=0}^{N-1} \sum_{k=1}^{N-1}   \mu [ (Y^{n})^T  \overline{\delta^{n,k} } Y^n ]
\\
&= -E_3 - E_4 - E_5.
\end{split}
\eeqn
Finally, since
\beqn
D^2 A(W) - D^2A(W^{n,m}) = D^2A(W^{n,-1}) - D^2A(W^{n,m}) = \sum_{k=0}^m\delta^{n,k},
\eeqn
we have 
\beqn
\begin{split}
& \sum_{n=0}^{N-1}\sum_{m=0}^{N-1} \mu [ (Y^{n})^T \mu (D^2A(W) - D^2A(W^{n,m})) Y^{n,m} ] 
\\
=\ & \sum_{n=0}^{N-1}\sum_{m=1}^{N-1} \mu \left[ (Y^{n})^T \sum_{k=0}^m \mu (\delta^{n,k}) Y^{n,m} \right] 
 +  \sum_{n=0}^{N-1} \mu [ (Y^{n})^T  \mu (\delta^{n,0}) Y^{n,m} ]  = E_6 + E_7.
\end{split}
\eeqn
This completes the proof for Proposition \ref{prop:sunklodas_multi}.
\end{proof}

\subsection{Proof for Theorem \ref{thm:main_multi_general} }

 We gather in the following lemma some  useful basic inequalities involving the spectral norm.

\begin{lem}\label{lem:spectral} For all $A,B \in \bR^{d\times d}$, $x \in \bR^d$, and $\alpha, \beta \in \{1,\ldots, d\}$:
\begin{itemize}
\item[(i)]{$\Vert A x \Vert \le \Vert A \Vert_s \Vert x \Vert$;} \smallskip
\item[(ii)]{$\Vert AB \Vert_s \le \Vert A \Vert_s \Vert B \Vert_s$;} \smallskip
\item[(iii)]{$|A_{\alpha\beta} | \le \Vert A \Vert_s \le \left( \max_{1 \le j \le d} \sum_{i=1}^d |A_{ij}| \right)^{\frac12} \left(\max_{1 \le i \le d} \sum_{j=1}^d |A_{ij}| \right)^{\frac12}$;} \smallskip
\item[(iv)]{$| \textnormal{tr} A | \le d \Vert A \Vert_s$;}\smallskip
\item[(v)]{$\Vert A \Vert_s = \sqrt{\lambda_{\max}(A^TA) } \le \sqrt{ \textnormal{tr} \, A^TA}$, where $\lambda_{\max}(A^TA)$ denotes the largest eigenvalue of the positive-semidefinite matrix $A^TA$.}
\end{itemize}
\end{lem}

\begin{lem} Let \(h: \, \bR^d \to \bR\) be three times differentiable with \(\Vert D^k h \Vert_{\infty} < \infty\) for \(1 \le k \le 3\) and let $A$ be the function \eqref{eq:stein_solution} that solves Stein's equation. Define $\delta^{n,k}(u)$ as in Proposition \ref{prop:sunklodas_multi}. Then, conditions (A2) and (A3) imply that, for all $0 \le n,m \le N-1$, the following two conditions hold:
\begin{itemize}
\item[(A2')]{Whenever \(u \in [0,1]\) and $m \le k \le N-1$,
\begin{align*}
| \mu[  (Y^{n})^T \delta^{n,k}(u) Y^{n,m}  ]| \le C_2 5M d^2 \Vert D^3 h \Vert_{\infty} \Vert b^{-1}  \Vert_s^3  \rho(m).
\end{align*}
} 
\item[(A3')]{Whenever \(2m \le k \le N-1\), 
\begin{align*}
| \mu [  (Y^{n})^T \, \overline{\delta^{n,k}} \, Y^{n,m}  ]| \le  C_3 5M d^2 \Vert D^3 h \Vert_{\infty} \Vert b^{-1}  \Vert_s^3  \rho(k -m).
\end{align*}}
\end{itemize}

\end{lem}

\begin{proof}  We denote
\begin{align*}
G_h(x,y) = G_h(x,y; s,t,z) = b^{-1} \left[ D^2h(s b^{-1} (  x + t y ) + z) - D^2h(sb^{-1}x + z) \right] b^{-1},
\end{align*}
where $(s,t,z) \in [0,1]^2 \times \bR^d$. Then,
\begin{align*}
|( b G_h b)_{\alpha,\beta}(x,y) | \le  \sup_{\xi} \Vert \nabla \partial_{\alpha,\beta} h(\xi) \Vert \Vert b^{-1} st y \Vert \le \sqrt{d} \Vert D^3 h \Vert_{\infty} \Vert b^{-1}  \Vert_s \Vert y \Vert,
\end{align*}
which together with Lemma \ref{lem:spectral}  implies
\begin{align}\label{eq:bound_delta}
\Vert ( b G_h b) (x,y) \Vert_s  \le d^2 \Vert D^3 h \Vert_{\infty} \Vert b^{-1}  \Vert_s \Vert y \Vert.
\end{align}
Hence,
\begin{align}\label{eq:G_sup}
\Vert G_h(x,y) \Vert_s \le d^2 \Vert D^3 h \Vert_{\infty} \Vert b^{-1}  \Vert_s^3 \Vert y \Vert.
\end{align}
Similarly we see that, for all $1 \le i \le 2d$,
\begin{align}\label{eq:G_partial}
\Vert \partial_i G_h(x,y) \Vert_s \le 2 d^2 \Vert D^3 h \Vert_{\infty} \Vert b^{-1}  \Vert_s^3.
\end{align}

For (A1') Suppose that \(m \le k \le N-1\). Recall from Lemma \ref{lem:stein_solution} that
\begin{align*}
&A(w)= -\int_0^{\infty} \! \left\{\int_{\bR^d} h(e^{-s}w + \sqrt{1 - e^{-2s}}\,z)\,\phi_\Sigma(z)\,dz - \Phi_{\Sigma}(h)\right\} \, ds
\end{align*}
solves the Stein equation \eqref{eq:stein_multi}. Since $h$ is three times differentiable with \(\Vert D^k h \Vert_{\infty} < \infty\) for \(1 \le k \le 3\), we can use dominated convergence to compute
\begin{align*}
D^2A(w) = - \int_0^{\infty} e^{-2s} \int_{\bR^d} D^2 h(e^{-s}w + \sqrt{1-e^{-2s} }z ) \phi_{\Sigma}(z) \, dz \, ds.
\end{align*}
Recall that, for a function $F : \bR^d \times B_d(0,4M+1) \to \bR^{d\times d}$, we denote
\begin{align*}
\Vert F  \Vert_\infty =  \sup \{ \Vert F(x,y) \Vert_s \, : \, (x,y)  \in \bR^d \times B_d(0,4M+1)  \}  
\end{align*}
and
\begin{align*}
\Vert \nabla F \Vert_\infty =  \max_{1 \le i \le 2d} \sup \{ \Vert \partial_i F(x,y) \Vert_s \, : \, (x,y)  \in \bR^d \times B_d(0,4M+1) \}.
\end{align*}
By Fubini's theorem,
\begin{align*}
& \mu[  (Y^{n})^T \delta^{n,k}(u) Y^{n,m}   ]\\
&= - \int_0^{\infty} e^{-2s} \int_{\bR^d}   \mu [  (Y^{n})^T (  D^2 h(e^{-s}(W^{n,k} + uY^{n,m} ) + \sqrt{1-e^{-2s}}z ) \\
&- D^2h(e^{-s}W^{n,k} + \sqrt{1-e^{-2s}}z  )) Y^{n,m}  ] \phi_{\Sigma}(z) \, dz \, ds \\
&= - \int_0^{\infty} e^{-2s} \int_{\bR^d}   \mu \left[  (\bar{f}^{n})^T G_h \left( \sum_{|i - n| > m} \bar{f}^i, \bar{f}^{n,m} ; e^{-s}, u, \sqrt{1 - e^{-2s}} z    \right) \bar{f}^{n,m}   \right] \phi_{\Sigma}(z) \, dz \, ds,
\end{align*}
so that an application of condition (A2) combined with \eqref{eq:G_sup} and \eqref{eq:G_partial}  yields
\begin{align*}
& |\mu[  (Y^{n})^T \delta^{n,k}(u) Y^{n,m}  ]|\\
&\le C_2 \left( \Vert G_h \Vert_{\infty} + \Vert \nabla G_h\Vert_{\infty} \right) \rho (m)\int_0^{\infty} e^{-2s} \int_{\bR^d}  \phi_{\Sigma}(z) \, dz \, ds \\
&= \frac{C_2}{2}  \left( \Vert G_h \Vert_{\infty} + \Vert \nabla G_h\Vert_{\infty} \right) \rho (m) \\
&\le  2 C_2 5M d^2 \Vert D^3 h \Vert_{\infty} \Vert b^{-1}  \Vert_s^3   \rho(m),
\end{align*}
which proves condition (A2'). The proof for condition (A3') is essentially the same which is why we omit it.
\end{proof}

We now proceed to show Theorem \ref{thm:main_multi_general}. Combining Lemma \ref{lem:stein_solution} with Proposition \ref{prop:sunklodas_multi} yields
\begin{align*}
|\mu[h(W)] - \Phi_{\Sigma}(h)| = |\mu[ \textnormal{tr} \Sigma D^2 A(W) - W^T \nabla A(W)]| \le \sum_{i=1}^7 |E_i|,
\end{align*}
where $A$ is given by \eqref{eq:stein_solution} and $E_i$ are as in Proposition \ref{prop:sunklodas_multi}. We bound each term $E_i$ separately, using conditions (A1), (A2') and (A3').

By condition (A2'),
\begin{align*}
|E_1| &=  \left| \sum_{n=0}^{N-1}\sum_{m=1}^{N-1} \int_0^1    | \mu[  (Y^{n})^T \delta^{n,m}(u) Y^{n,m}   ] \, du \right| 
\le  \sum_{n=0}^{N-1}\sum_{m=1}^{N-1} C_2 5M d^2 \Vert D^3 h \Vert_{\infty} \Vert b^{-1}  \Vert_s^3  \rho(m).\\
&= C_2 5M d^2 \Vert D^3 h \Vert_{\infty} N \Vert b^{-1}  \Vert_s^3  \sum_{m=1}^{N-1} \rho(m).
\end{align*}
Moreover, 
\begin{align*}
|E_2| &= \left| \sum_{n=0}^{N-1} \int_0^1   | \mu[ (Y^{n})^T \delta^{n,0}(u) Y^n ]  \, du \right|
\le \sum_{n=0}^{N-1} \int_0^1     \mu[ \Vert Y^n \Vert \Vert  \delta^{n,0}(u) \Vert_s \Vert Y^n \Vert ]   \, du\\
&\le N \Vert b^{-1} \Vert_s^2 4M^2 \cdot d^2 \Vert D^3 h \Vert_{\infty} \Vert b^{-1}  \Vert_s (4M + 1) \\
&\le 20 M^3 d^2\Vert D^3 h \Vert_{\infty} N  \Vert b^{-1}  \Vert_s^3,
\end{align*}
where \eqref{eq:bound_delta} was used in the third inequality.

For $E_3$ first note that 
\begin{align*}
\mu [  (Y^{n})^T \mu( \delta^{n,k})  Y^{n,m}  ] = \text{tr} \,  \mu( Y^n \otimes Y^{n,m} ) \,  \mu(  \delta^{n,k} ),
\end{align*}
so that Lemma \ref{lem:spectral}  and condition (A1) can be used to obtain
\begin{align}
| \mu [  (Y^{n})^T \mu( \delta^{n,k})  Y^{n,m}  ] | &\le d \Vert  \mu( Y^n \otimes Y^{n,m} )   \Vert_s \Vert \mu(  \delta^{n,k} ) \Vert_s \notag \\
&\le d \Vert b^{-1} \Vert_s^2 \Vert  \mu( \bar{f}^n \otimes \bar{f}^{n,m} )   \Vert_s \Vert \mu(  \delta^{n,k} ) \Vert_s \notag \\
&\le d^2 \Vert b^{-1} \Vert_s^2 \Vert  C_1 \rho(m) \cdot  d^2 \Vert D^3 h \Vert_{\infty} \Vert b^{-1}  \Vert_s (4M + 1) \notag \\
&\le C_1 5M d^4 \Vert D^3 h \Vert_{\infty} \Vert b^{-1} \Vert_s^3  \rho(m). \label{eq:bound_a1}
\end{align}
Combinining \eqref{eq:bound_a1} with an application of condition (A2') yields
\begin{align*}
|E_3| & = \left| \sum_{n=0}^{N-1}\sum_{m=1}^{N-1} \sum_{k=m+1}^{2m} \mu [  (Y^{n})^T \, \overline{\delta^{n,k}} \, Y^{n,m}  ] \right| \\
&\le \sum_{n=0}^{N-1}\sum_{m=1}^{N-1} \sum_{k=m+1}^{2m} (  | \mu [   (Y^{n})^T \delta^{n,k}  Y^{n,m}  ]| + |\mu [  (Y^{n})^T \mu( \delta^{n,k})  Y^{n,m}  ] | ) \\
&\le  \left(  C_2  d^2   +  C_1  d^4  \right)  5M \Vert D^3 h \Vert_{\infty} \Vert b^{-1} \Vert_s^3  N \sum_{m=1}^{N-1} m \rho(m).
\end{align*}

Condition (A3') is used to bound \(E_4\) and \(E_5\):
\begin{align*}
|E_4| &= \left|  \sum_{n=0}^{N-1}\sum_{m=1}^{N-1} \sum_{k=2m+1}^{N-1}    | \mu [  (Y^{n})^T \, \overline{\delta^{n,k}} \, Y^{n,m}  ] \right| \\
&\le   C_3 5M d^2 \Vert D^3 h \Vert_{\infty}  N  \Vert b^{-1}  \Vert_s^3  \sum_{m=1}^{N-1} m \rho (m),
\end{align*}
and
\begin{align*}
|E_5| &= \left| \sum_{n=0}^{N-1} \sum_{k=1}^{N-1}  |\mu [  (Y^{n})^T \, \overline{\delta^{n,k}} \, Y^n  ] \right|
 \le   C_3 5M d^2 \Vert D^3 h \Vert_{\infty} N \Vert b^{-1}  \Vert_s^3    \sum_{m=1}^{N-1}  \rho(m).
\end{align*}

Again by  \eqref{eq:bound_a1}, 
\begin{align*}
|E_6| =  \left| \sum_{n=0}^{N-1}\sum_{m=1}^{N-1} \sum_{k=0}^m   \mu [  (Y^{n})^T    \mu( \delta^{n,k} )  Y^{n,m}  ] \right| 
\le   C_1 5M d^4 \Vert D^3 h \Vert_{\infty} N \Vert b^{-1} \Vert_s^3 \sum_{m=1}^{N-1} m  \rho(m).
\end{align*}

Finally,
\begin{align*}
|E_7| &= \left|  \sum_{n=0}^{N-1}   \mu [  (Y^{n})^T   \mu( \delta^{n,0} )  Y^n  ] \right|
\le  \sum_{n=0}^{N-1}  \mu[ \Vert Y^n \Vert^2  ] \Vert \mu(\delta^{n,0})  \Vert_s   \\
&\le 20 M^3  d^2 \Vert D^3 h \Vert_{\infty}  N \Vert b^{-1}  \Vert_s^3.
\end{align*}

Gathering the foregoing upper bounds we obtain 
\begin{align*}
&|\mu[h(W)] - \Phi_{\Sigma}(h)|\\
&\le N  \Vert b^{-1}  \Vert_s^3 \sum_{m=1}^{N-1} m\rho(m)  \Vert D^3 h \Vert_{\infty} \Biggl(   C_2 5M d^2 + 20 M^3 d^2 + \left(  C_2  d^2   +  C_1  d^4  \right)  5M \\
&+  2 C_3 5M d^2 +  C_1 5M d^4  + 20 M^3  d^2\Biggr)\\
&\le  N  \Vert b^{-1}  \Vert_s^3 \sum_{m=1}^{N-1} m\rho(m)  \Vert D^3 h \Vert_{\infty} M^3 d^4 10 ( C_1 + C_2 + C_3 + 4  ).
\end{align*}
The proof for Theorem  \ref{thm:main_multi_general} is complete.

\subsection{Proof for Theorem \ref{thm:main_1d}}  Since the proof for Theorem \ref{thm:main_1d} is very similar to the proof for Theorem  \ref{thm:main_multi_general}, we omit most of the details and only give an outline, emphasizing differences between the two proofs. 

Now $b^2 =  \text{Var}_\mu( \sum_{i < N} \bar{f}^i ) > 0$ so that $\text{Var}_\mu(W) = \mu(W^2) = 1$. Then the univariate Stein equation is defined by 
\beq\label{eq:stein_univ}
 A'(w) - wA(w) = h(w) - \Phi_{1}(h),
\eeq
where $w \in \bR$. Note that the order of \eqref{eq:stein_univ} is one smaller than the order of the multivariate Stein equation \eqref{eq:stein_multi}. We have the following result regarding the regularity of $A$:

\begin{lem}[See~\cite{Chen_etal_2011}]\label{lem:univ_stein_sol} Whenever $h : \bR \to \bR$ is Lipschitz continuous with $\textnormal{Lip}(h) \le 1$ the solution $A : \bR \to \bR$ to \eqref{eq:stein_univ} belongs to the class $\mathscr{F}_1$ consisting of all differentiable functions with an absolutely continuous derivative, satisfying the bounds
\beqn
\|A\|_\infty \le 2, 
\quad
\|A'\|_\infty \le \sqrt{2/\pi}, \, 
\quad\text{and}\quad
\|A''\|_\infty \le 2.
\eeqn
\end{lem}

The lemma implies that
\begin{align*}
d_{\mathscr{W}}(W,Z) \le \sup_{A \in \mathscr{F}_1} |\mu[A'(W) - WA(W) ] |,
\end{align*}
where $Z \sim \cN(0,1)$. Next $\mu[A'(W) - WA(W) ]$ is decomposed precisely as in Proposition 4 of \cite{sunklodas2007}.  The decomposition is the same as that given in Proposition \ref{prop:sunklodas_multi} except that  $\delta^{n,m}(u)$ there is replaced with
\begin{align*}
\delta^{n,m}(u) = A'(W^{n,m} + u\, Y^{n,m}) - A'(W^{n,m}).
\end{align*}
Then
\begin{align*}
\delta_{n,m}(u) = G_u \Biggl( \sum_{|i - n| > m} \bar{f}^i, \bar{f}^{n,m}   \Biggl),
\end{align*}
where 
\begin{align*}
G_u(x,y) =   A'\left( b^{-1}x + b^{-1}u y \right) - A'\left( b^{-1} x  \right).
\end{align*}
By Lemma \ref{lem:univ_stein_sol}
\begin{align*}
| G_u(x,y) | \le \text{Lip}(A')| b^{-1}u y | \le 2 b^{-1} |y|
\end{align*}
and
\begin{align*}
\text{Lip}(G_u) \le 4 b^{-1}.
\end{align*}
Hence, conditions (B2) and (B3) can be applied with $G_u \upharpoonright ( \bR \times B_1(0, 4M+ 1))$ as in the proof for Theorem \ref{thm:main_multi_general}. Using also condition (B1) we obtain bounds to each of the terms  $E_i$ appearing in the univariate version of Proposition  \ref{prop:sunklodas_multi}, which then lead to the upper bound \eqref{eq:scalar_bound}.

\subsection{Proof for Theorem \ref{thm:main_self}} From now on we assume that $\operatorname{Cov}_\mu \left(\sum_{i=0}^{N-1} \bar f^i \right)$ is positive definite and take 
\begin{align*}
b = \left[ \operatorname{Cov}_\mu \left(\sum_{i=0}^{N-1} \bar f^i \right)  \right]^{1/2} ,
\end{align*}
in which case $\Sigma = \mu(W \otimes W) = I_{d\times d}$. By Lemma \ref{lem:spectral},
\begin{align*}
\Vert b^{-1} \Vert^2_s = \lambda_{\max}   \left( \left[ \operatorname{Cov}_\mu \left(\sum_{i=0}^{N-1} \bar f^i\right)  \right]^{-1}  \right) = \lambda_{\min}^{-1},
\end{align*}
where we recall that $\lambda_{\min}$ is the least eigenvalue of $\operatorname{Cov}_\mu(\sum_{i=0}^{N-1} \bar f^i)$. 

By Lemma \ref{lem:solution_self},
\begin{align*}
d_{\mathscr{W}}(W, Z) \le \sup_{A \in \cA} |  \mu[ \tr  D^2A(W) - W^T  \nabla A(W) ]|,
\end{align*}
where $Z \sim \cN(0, I_{d\times d})$ and $\cA$ denotes the class of all $C^2$ functions satisfying \eqref{eq:log_lip}. The proof then proceeds as follows. First we decompose $\mu[ \tr  D^2A(W) - W^T  \nabla A(W) ] = \sum_{i=1}^7 E_i$ using Proposition \ref{prop:sunklodas_multi}, which reduces the proof to bounding each term $E_i$ for functions $A \in \cA$. For example, to obtain an upper bound on $E_1$ we have to control the integral
\begin{align*}
 \int_0^1  \, \mu[  (Y^{n})^T \delta^{n,m}(u) Y^{n,m}   ] \, du,
\end{align*}
where we recall that $\delta^{n,m}(u) = D^2A(W^{n,m} + u\, Y^{n,m}) - D^2A(W^{n,m})$. To this end we will describe a class $\cG$ of regular functions $G : \bR^d \times \bR^d \to \bR^{d \times d}$ such that 
\begin{align}\label{eq:e_1_integral}
 &\int_0^1  \, \mu[  (Y^{n})^T \delta^{n,m}(u) Y^{n,m}   ] \, du 
 =  \mu \left[  ( \bar{f}^{n})^T   G \left( \sum_{|i - n| > k} \bar{f}^i, \bar{f}^{n,k}   \right)  \bar{f}^{n,m}   \right].
\end{align}
The integral on the right is bounded by condition (C2), provided that $G$ is a $C^1$-function. This might not be the case, since functions in $\cG$ will have the same regularity as the second derivatives of functions in $\cA$, which according to Lemma \ref{lem:solution_self} is  Lipschitz up to a logarithmic factor. But we can approximate such functions by  $C^\infty$-functions, which in combination with condition (C2) then leads to an upper bound on \eqref{eq:e_1_integral} and consequently on $E_1$. The other terms $E_i$ will be treated similarly. We now proceed to detail the foregoing argument.

We denote by $\cG$ the collection of all functions $G : \bR^d \times \bR^d \to \bR^{d \times d}$  that satisfy the following upper bounds: 
\begin{align*}
\sup_{x} \Vert G(x,y) \Vert_s \le  K  \lambda_{\min}^{-\frac32} (1+  \log N  ) ( \Vert y \Vert^2 + 1) \hspace{0.5cm} \forall y \in \bR^d; 
\end{align*}
\begin{align*}
&\sup_{y } \Vert  G(a,y) - G(a',y)  \Vert_s \\
&\le K \lambda_{\min}^{-\frac32} (1 + \log N) \Vert a - a' \Vert^{} ( 1 + |\log \Vert a - a' \Vert   | )  \hspace{0.5cm} \forall a,a' \in \bR^d;
\end{align*}
\begin{align*}
&\sup_{x } \Vert  G(x,a) - G(x,a')  \Vert_s \\
&\le K \lambda_{\min}^{-\frac32} (1 + \log N) \Vert a - a' \Vert^{} ( 1 + |\log \Vert a - a' \Vert   | )  \hspace{0.5cm} \forall a,a' \in \bR^d.
\end{align*}
where $K =  2C_\# + \sqrt{d}4M + 2 $  and $C_\#$ is the constant from Lemma \ref{lem:solution_self} with $\eta = 1$. \medskip

\begin{lem}\label{lem:class_G} Assume $\lambda_{\min} > 1$. Then, given any $A \in \cA$ and $0 \le n,m \le N-1$, there is a function $G_u : \bR^d \times \bR^d \to \bR^{d \times d}$ satisfying
\begin{align}\label{eq:func_class}
G_u \left( \sum_{|i - n| > m} \bar{f}^i, \bar{f}^{n,m}   \right) =     b^{-1} \,   \delta^{n,m}(u) b^{-1},
\end{align}
where $\delta^{n,m}(u)$ is defined as in Proposition \ref{prop:sunklodas_multi}, such that
\begin{align*}
G_{1}  \in \cG\hspace{0.2cm} \text{and} \hspace{0.2cm} G = \int_0^1 G_{u} \, du \in \cG.
\end{align*}
\end{lem}

\begin{proof} It is easy to see that \eqref{eq:func_class} holds with $G_u(x,y)$ defined as 
\begin{align*}
 b^{-1} \left[ D^2A\left( b^{-1}x + b^{-1}u y \right) - D^2A\left( b^{-1} x  \right) \right]  b^{-1},
\end{align*}
We show that $G \in \cG$ and leave the similar verification of $G_1 \in \cG$ to the reader. 

Observe that, by Lemma \ref{lem:spectral} and \eqref{eq:log_lip},
\begin{align}\label{eq:aux_g}
\Vert G_u(x,y) \Vert_s \le  \Vert b^{-1} \Vert_s^2 \Vert u b^{-1}y \Vert ( C_\# + | \log \Vert u b^{-1} y \Vert | ) 
\end{align}
holds for all $x \in \bR^d$, $y \in \bR^d \setminus \{ 0 \}$, and $u \in (0,1]$. Then assume (as we may) that $y \neq 0$. We use  \eqref{eq:aux_g} and $\Vert b^{-1} \Vert_s = \lambda_{\min}^{-1/2} <  1$ to obtain 
\begin{align*}
\Vert G(x,y) \Vert_s 
&\le  \int_0^1 \Vert G_u(x,y) \Vert_s \, du  \le \Vert b^{-1} \Vert_s^2 \int_0^1 \Vert u b^{-1}y \Vert ( C_\# + | \log \Vert u b^{-1} y \Vert | ) \, du \\
&\le \Vert b^{-1} \Vert_s^3 \Vert  y \Vert ( C_\# + 1 + | \log \Vert  b^{-1} y \Vert | ) \\
&\le \lambda_{\min}^{-\frac32} ( 1 + \Vert y \Vert^2 )   \left(  C_\# + 2 + \log \Vert b^{-1} \Vert_s^{-1}   \right).
\end{align*}
Since $1\le \Vert b \Vert_s \Vert b^{-1}  \Vert_s$, 
\begin{align}\label{eq:b_spec_ineq}
\Vert b^{-1}  \Vert^{-1}_s \le \Vert b \Vert_s \le  \left(\text{tr} \operatorname{Cov}_\mu \left(\sum_{i=0}^{N-1} \bar f^i \right) \right)^{1/2}  \le \sqrt{d} 2 M N,
\end{align}
where we used Lemma \ref{lem:spectral}.  Hence,
\begin{align*}
\Vert G(x,y) \Vert_s &\le  \lambda_{\min}^{-\frac32} ( 1 + \Vert y \Vert^2 ) ( \log N + 1)  (  C_\# + 2 + \sqrt{d}2M  ).
\end{align*}

Next let $a,a',y \in \bR^d$. Then,
\begin{align*}
\Vert G(a,y) - G(a',y) \Vert_s &\le \int_0^1 \Vert G_u(a,y) - G_u(a',y) \Vert_s  \, du \\
&\le 2 \Vert b^{-1} \Vert_s^2 \left( \Vert b^{-1}( a - a') \Vert ( C_\# + | \log \Vert b^{-1}(a - a' ) \Vert | )   \right) \\
&\le 2 \Vert b^{-1} \Vert_s^3 \left(  \Vert a - a' \Vert ( C_\# + \log( \sqrt{d} 2MN) )  + \Vert a - a' \Vert | \log \Vert a - a' \Vert |  \right) \\
&\le 2 \lambda_{\min}^{-\frac32} \Vert a - a' \Vert ( 1 + | \log \Vert a - a' \Vert |   ) ( C_\# + \sqrt{d}2M + \log N + 1 ) \\
&\le  \lambda_{\min}^{-\frac32} \Vert a - a' \Vert ( 1 + | \log \Vert a - a' \Vert |   )( \log N + 1) ( 2C_\# + \sqrt{d}4M + 2 ),
\end{align*}
where \eqref{eq:aux_g} was used in the second inequality, and \eqref{eq:b_spec_ineq} in the third inequality.

Finally, for all $a,a',x \in \bR^d$,
\begin{align*}
\Vert G(x,a) - G(x,a') \Vert_s &\le \int_0^1 \Vert G_u(x,a) - G_u(x,a') \Vert_s  \, du\\
&\le \Vert b^{-1} \Vert_s^2 \int_0^1 \Vert u b^{-1} ( a - a' ) \Vert ( C_\# + | \log \Vert u b^{-1} (a - a' ) \Vert | ) \, du \\
&\le \Vert b^{-1} \Vert_s^3  \Vert  a - a' \Vert ( C_\# + 1 +   \log \Vert b^{-1} \Vert^{-1} +| \log \Vert a - a'  \Vert | ) \\
&\le  \Vert b^{-1} \Vert_s^3  \Vert  a - a' \Vert ( 1 + | \log \Vert a - a'  \Vert | ) ( C_\# + \sqrt{d}2M + \log N + 1 ) \\
&\le \Vert b^{-1} \Vert_s^3  \Vert  a - a' \Vert ( 1 + | \log \Vert a - a'  \Vert | ) (\log N + 1) ( C_\# + \sqrt{d}2M + 1),
\end{align*}
where \eqref{eq:b_spec_ineq} was used in the second last inequality. This completes the proof for $G \in \cG$.
\end{proof}

The following lemma is established by a standard approximation argument. See Appendix \ref{sec:approx-lemma-pf} for the proof.

\begin{lem}\label{lem:new_cond} Conditions (C2) and (C3) imply that, for all $0 \le n,m \le N-1$, the following two conditions hold:
\begin{itemize}
\item[(C2')]{Whenever $m \le k \le N-1$ and $G \in \cG$,
\begin{align*}
&\left| \mu \left[  ( \bar{f}^{n})^T   G \left( \sum_{|i - n| > k} \bar{f}^i, \bar{f}^{n,k}   \right)  \bar{f}^{n,m}   \right] \right|  \notag \\
&\le  C_2'  \lambda_{\min}^{-\frac32} (1 + \log N) (1 + \log (\rho(m)^{-1}) )  \rho ( m ),
\end{align*} 
where
\begin{align*}
C_2' =  C_2 d 4^d K673 M^2   \left( 1 + \tfrac{1}{\rho(0)} \right)   ( 2 \rho(0) + 1 ).
\end{align*}
\smallskip
}
\item[(C3')]{Whenever $2m \le k \le N-1$ and $G \in \cG$,
\begin{align*}
&\left| \mu \left[  ( \bar{f}^{n})^T   \overline{G \left( \sum_{|i - n| > k} \bar{f}^i, \bar{f}^{n,k}   \right)}  \bar{f}^{n,m}   \right] \right|  \notag \\
&\le C_3'  \lambda_{\min}^{-\frac32} (1 + \log N) (1 + \log (\rho(k-m)^{-1} ))  \rho ( k-m ),
\end{align*}
where
\begin{align*}
C_3' =  C_3 d 4^d K673 M^2   \left( 1 + \tfrac{1}{\rho(0)} \right)   ( 2 \rho(0) + 1 ).
\end{align*}}
\end{itemize}
\end{lem}

We proceed to bound the terms $E_i$ in Proposition \ref{prop:sunklodas_multi} using conditions (C1), (C2') and (C3'). Let $G_u$ be a function as in Lemma \ref{lem:class_G} and set $G = \int_0^1 G_u \, du$. Then for \(E_1\) we have by condition (C2') the upper bound
\begin{align*}
|E_1| &= \left| \sum_{n=0}^{N-1}\sum_{m=1}^{N-1}  \mu\left[  ( \bar{f}^n )^T  G \left( \sum_{|i - n| > m} \bar{f}^i, \bar{f}^{n,m}   \right) \bar{f}^{n,m}  \right]\right| \\
&\le  \sum_{n=0}^{N-1}\sum_{m=1}^{N-1} C_2'  \lambda_{\min}^{-\frac32} (1 + \log N) (1 + \log (\rho(m)^{-1}) )  \rho ( m ) \\
&=  N(1 + \log N)\lambda_{\min}^{-\frac32} C_2'  \sum_{m=1}^{N-1}    (1 + \log (\rho(m)^{-1}) )  \rho ( m ).
\end{align*}

Since $G \in \cG$,
\begin{align*}
|E_2| &= \left| \sum_{n=0}^{N-1}   \mu\left[  (\bar{f}^{n})^T G \left( \sum_{|i - n| > 0} \bar{f}^i, \bar{f}^{n}   \right)  \bar{f}^{n}   \right]\right| \\
&\le N 4M^2 K  \lambda_{\min}^{-\frac32} (1+  \log N  ) ( (2M)^2 + 1) \\
&\le N(1 + \log N)\lambda_{\min}^{-\frac32} K 20M^4.
\end{align*}

For $E_3$ we note that 
\begin{align*}
\mu [  (Y^{n})^T \mu( \delta^{n,k})  Y^{n,m}  ] = \text{tr} \,  \mu( \bar{f}^n \otimes \bar{f}^{n,m} ) \,  \mu\left[G_1 \left( \sum_{|i - n| > k} \bar{f}^i, \bar{f}^{n,k}   \right)\right].
\end{align*}
Hence, by Lemma \ref{lem:spectral}  and condition (C1),
\begin{align}\label{eq:proof_main_3_eq1}
&|\mu [  (Y^{n})^T \mu( \delta^{n,k})  Y^{n,m}  ]| \notag \\
&\le d \Vert \mu( \bar{f}^n \otimes \bar{f}^{n,m} )\Vert_s    \Vert G_1 \Vert_\infty \notag \\
&\le d^2 C_1 \rho(m) K  \lambda_{\min}^{-\frac32} (1+  \log N  ) ( (4M + 1 )^2 + 1).
\end{align}
Combining \eqref{eq:proof_main_3_eq1} with condition  (C2') implies the upper bound
\begin{align*}
|E_3| & \le \sum_{n=0}^{N-1}\sum_{m=1}^{N-1} \sum_{k=m+1}^{2m} | \mu [  (Y^{n})^T \, \overline{\delta^{n,k}} \, Y^{n,m}  ] |\\
&\le \sum_{n=0}^{N-1}\sum_{m=1}^{N-1} \sum_{k=m+1}^{2m} (  | \mu [  (Y^{n})^T \delta^{n,k}  Y^{n,m}  ]| + |\mu [  (Y^{n})^T \mu( \delta^{n,k})  Y^{n,m}  ] | ) \\
&\le \sum_{n=0}^{N-1}\sum_{m=1}^{N-1} \sum_{k=m+1}^{2m} \left|  \mu\left[  ( \bar{f}^n )^T  G_1 \left( \sum_{|i - n| > k} \bar{f}^i, \bar{f}^{n,k}   \right) \bar{f}^{n,m}  \right] \right| \\
&+ \sum_{n=0}^{N-1}\sum_{m=1}^{N-1} \sum_{k=m+1}^{2m} d^2 C_1 \rho(m) K  \lambda_{\min}^{-\frac32} (1+  \log N  ) ( ( 4M + 1 )^2 + 1)\\
&\le  N(1 + \log N )\lambda_{\min}^{-\frac{3}{2}} \Biggl(  \sum_{m=1}^{N-1}   C_2'  (1 + \log (\rho(m)^{-1} )) m \rho ( m )\\ &+ \sum_{m=1}^{N-1} d^2 C_1 m\rho(m) K  ( \Vert 4M + 1 \Vert^2 + 1) \Biggr) \\
&\le N(1 + \log N )\lambda_{\min}^{-\frac{3}{2}} (2C_2' + d^2C_1K52M^2 ) \sum_{m=1}^{N-1}  (1 + \log (\rho(m)^{-1}) ) m \rho ( m ).
\end{align*}

Next condition (C3') is used to bound \(E_4\) and \(E_5\):
\begin{align*}
|E_4| &= \left|  \sum_{n=0}^{N-1}\sum_{m=1}^{N-1} \sum_{k=2m+1}^{N-1}  \mu [  (\bar{f}^n )^T \, \overline{G_1 \left( \sum_{|i - n| > k} \bar{f}^i, \bar{f}^{n,k}   \right) } \, \bar{f}^{n,m}  ] \right| \\
&\le N \sum_{m=1}^{N-1} \sum_{k=2m+1}^{N-1} C_3'  \lambda_{\min}^{-\frac32} (1 + \log N) (1 + \log (\rho(k-m)^{-1} ))  \rho ( k-m )\\
&\le  N(1 + \log N )\lambda_{\min}^{-\frac{3}{2}} C_3'  \sum_{m=1}^{N-1}  (1 + \log (\rho(m)^{-1} )) m  \rho ( m ),
\end{align*}
and
\begin{align*}
|E_5| &= \left| \sum_{n=0}^{N-1} \sum_{k=1}^{N-1}  \mu \left[  (\bar{f}^{n})^T \,   \overline{G_1 \left( \sum_{|i - n| > k} \bar{f}^i, \bar{f}^{n,k}   \right) }  \, \bar{f}^n  \right] \right| \\
&\le  N    (1 + \log N )  \lambda_{\min}^{-\frac{3}{2}} C_3'   \sum_{m=1}^{N-1}  (1 + \log (\rho(m)^{-1} ))  \rho ( m ).
\end{align*}

Again by \eqref{eq:proof_main_3_eq1},
\begin{align*}
|E_6| &= \left|  \sum_{n=0}^{N-1}\sum_{m=1}^{N-1} \sum_{k=0}^m \mu \left[  (Y^{n})^T    \mu( \delta^{n,k} )  Y^{n,m}  \right] \right| \\
&\le \sum_{n=0}^{N-1}\sum_{m=1}^{N-1} \sum_{k=0}^m d^2 C_1 \rho(m) K  \lambda_{\min}^{-\frac32} (1+  \log N  ) ( ( 4M + 1 )^2 + 1).\\
&= N   (1+  \log N  )  \lambda_{\min}^{-\frac{3}{2}}  d^2C_1K  26M^2    \sum_{k=1}^{N-1} m \rho(m).
\end{align*}

Finally,
\begin{align*}
|E_7| &= \left|  \sum_{n=0}^{N-1}   \mu \left[  ( \bar{f}^{n})^T   \mu\left( G_1 \left( \sum_{|i - n| > 0} \bar{f}^i, \bar{f}^{n}   \right) \right)     \bar{f}^n  \right] \right| \\
&\le N 4M^2   \Vert G_1 \Vert_\infty 
\le  N(1+  \log N  ) \lambda_{\min}^{-\frac32}  104 M^4 K .
\end{align*}

Recall that, by Lemma \ref{lem:solution_self},
\begin{align*}
d_{\mathscr{W}}(W,Z)  \le \sup_{A \in \cA} |\mu[ \tr  D^2A(W) - W^T  \nabla A(W) ]|.
\end{align*}
Hence, Proposition \ref{prop:sunklodas_multi} together with the above bounds implies
\begin{align*}
d_{\mathscr{W}}(W,Z) &\le N ( 1 + \log N ) \lambda_{\min}^{-\frac{3}{2}} \biggl[   (3C_2' + 2C_3' + d^2C_1K52M^2 ) \sum_{m=1}^{N-1}  (1 + \log (\rho(m)^{-1} )) m \rho ( m )\\
&+ K 20M^4  + d^2C_1K  26M^2    \sum_{k=1}^{N-1} m \rho(m) + 104 M^4 K 
\biggr] \\
&\le N ( 1 + \log N ) \lambda_{\min}^{-\frac{3}{2}} \sum_{m=1}^{N-1}  (1 + \log (\rho(m)^{-1} ) ) m \rho ( m )  \biggl[   3C_2' + 2C_3'  \\
&+ d^2C_1K  78M^2    + 124 M^4 K 
\biggr].
\end{align*}
The proof for Theorem \ref{thm:main_self} is complete.

\appendix

\section{Proof for Lemma \ref{lem:new_cond} }\label{sec:approx-lemma-pf}

Let us define the mollifier $\eta : \bR^d \to \bR$ by $\eta(x) = c \varphi(1 - \Vert x \Vert^2 )$ where
\begin{align*}
\varphi(t) = \begin{cases} 
e^{-1/t^2} &\text{if} \hspace{0.5cm} t > 0, \\
0 &\text{if} \hspace{0.5cm} t \le 0,
\end{cases}
\end{align*}
and $c > 0$ is such that $\int_{\bR^d} \eta(x) \, dx = 1$. Then  
\begin{align}
c^{-1} &=  \int \varphi(1 - \Vert x \Vert^2) \, dx  \ge \int_{B_d(0,1/2)} \varphi(1 - \Vert x \Vert^2 )  \, dx \ge \varphi(\tfrac34) m(B_d(0,\tfrac12)) \notag \\
&\ge e^{-2} (\tfrac12)^d m(B_d(0,1)). \label{eq:c_lb}
\end{align}

Let $G \in \cG$. We approximate the components of $G$ by convolutions $G^{\ve}_{\alpha,\beta} : \bR^d \times \bR^d \to \bR$,
\begin{align*}
G^{\ve}_{\alpha, \beta}(x) = \int_{\bR^d \times \bR^d} G_{\alpha,\beta}(y) j_{\ve}(x-y) \, dy = \int_{\bR^d \times \bR^d} G_{\alpha,\beta}(x-\ve y) j(y) \, dy,
\end{align*}
where $x = (x_1,x_2) \in \bR^d \times \bR^d$, $j_\ve(x) = \ve^{-2d} j( x / \ve)$, and $j(x) =\eta(x_1) \eta(x_2)$.

 For all $\alpha, \beta \in \{1,\ldots, d\}$, $x = (x_1,x_2) \in \bR^d \times \bR^d$, and $\ve \in (0,1)$:
\begin{align*}
&| G_{\alpha,\beta}^{\ve}(x) - G_{\alpha,\beta}(x) |  \le \int_{\bR^d \times \bR^d} |G_{\alpha,\beta}(x-\ve y) - G_{\alpha,\beta}(x)  | j(y) \, dy \\
&\le \int_{B_d(0,1) \times B_d(0,1)} \Vert G(x-\ve y) - G(x)  \Vert_s j(y) \, dy \\
&\le K \lambda_{\min}^{-\frac32} (1 + \log N ) \int_{B_d(0,1) \times B_d(0,1)} \left[ \Vert \ve y_1 \Vert (1 + \log \Vert \ve y_1 \Vert^{-1} ) + \Vert \ve y_2 \Vert (1 + \log \Vert \ve y_2 \Vert^{-1} )   \right] j(y) \, dy \\
&\le K \lambda_{\min}^{-\frac32} (1 + \log N ) 6 \ve \log \ve^{-1}.
\end{align*}
Lemma \ref{lem:spectral} was used in the second inequality and $G \in \cG$ in the third inequality. It follows by Lemma \ref{lem:spectral} that 
\begin{align}\label{eq:approx_1}
\Vert G^{\ve}(x) - G(x) \Vert_s  \le d K \lambda_{\min}^{-\frac32} (1 + \log N ) 6 \ve \log \ve^{-1}.
\end{align}

Since $G \in \cG$, 
\begin{align*}
| G_{\alpha,\beta}^{\ve}(x) | &\le \sup_{y_i : \Vert y_i \Vert \le \Vert x_i \Vert + \ve} | G_{\alpha,\beta}(y_1,y_2) | \le \sup_{y_i : \Vert y \Vert \le \Vert x_i \Vert + \ve}  \Vert G(y_1,y_2) \Vert_s \\
&\le K  \lambda_{\min}^{-\frac32} (1+  \log N  ) ( (\Vert x_2 \Vert + \ve )^2 + 1),
\end{align*}
so that Lemma \ref{lem:spectral} implies
\begin{align}\label{eq:approx_2}
\Vert G^{\ve}(x) \Vert \le dK  \lambda_{\min}^{-\frac32} (1+  \log N  ) ( (\Vert x_2 \Vert + \ve )^2 + 1).
\end{align}

Since
\begin{align*}
G_{\alpha,\beta}^{\ve}(x) - G_{\alpha,\beta}^{\ve}(y) &= \int_{\bR^d \times \bR^d} G_{\alpha,\beta}(z) ( j_{\ve}(x-z) - j_{\ve}(y-z)) \, dz \\
&=\int_{\bR^d \times \bR^d} (G_{\alpha,\beta}(z) - G_{\alpha,\beta}(x)) ( j_{\ve}(x-z) - j_{\ve}(y-z)) \, dz,
\end{align*}
we have
\begin{align*}
\partial_iG_{\alpha,\beta}^{\ve}(x) &= \int_{\bR^d \times \bR^d}  (G_{\alpha,\beta}(z) - G_{\alpha,\beta}(x)) (\partial_i j_{\ve}) (x-z) \, dz \\
&= \int_{B_d(x, \ve) \times B_d(x, \ve)} \Vert G(z) - G(x) \Vert_s \ve^{-2d-1}  \partial_i j ( \tfrac{x-z}{\ve})  \, dz. 
\end{align*}
An easy computation shows that $|\partial_i j(x) | \le 12 c^2$. Using this, $G \in \cG$, and \eqref{eq:c_lb} we obtain for all $\ve \in (0,1)$ the upper bound
\begin{align*}
|\partial_iG_{\alpha,\beta}^{\ve}(x)| &\le \int_{B_d(x, \ve) \times B_d(x, \ve)} \Vert G(z) - G(x) \Vert_s \ve^{-2d-1} \left| \partial_i j\left( \tfrac{x-z}{\ve}\right) \right| \, dz \\
&\le \ve^{-2d-1} \left( 2K \lambda_{\min}^{-\frac32} (1 + \log N) \ve (1 + \log \ve^{-1}) \cdot 12c^2 \cdot m(B_d(x,\ve))^2  \right) \\
&= 2K \lambda_{\min}^{-\frac32} (1 + \log N)  (1 + \log \ve^{-1}) \cdot 12c^2 \cdot m(B_d(0,1))^2  \\
&\le \log (\ve^{-1}) 48K \lambda_{\min}^{-\frac32}  \cdot 12 \cdot  4^d = \log (\ve^{-1}) 576K \lambda_{\min}^{-\frac32} \cdot  4^d.
\end{align*}
Hence, by Lemma \ref{lem:spectral},
\begin{align}\label{eq:approx_3}
\Vert \partial_iG^{\ve}(x) \Vert_s \le \log (\ve^{-1}) 576K \lambda_{\min}^{-\frac32}   d4^d, \hspace{0.5cm} 1 \le i \le 2d,
\end{align}
where $\partial_iG^{\ve}(x) = [\partial_iG^{\ve}_{\alpha,\beta}(x) ]_{\alpha,\beta}$.

We combine \eqref{eq:approx_1}-\eqref{eq:approx_3} with condition (C2) to obtain
\begin{align*}
&\left| \mu \left[  ( \bar{f}^{n})^T   G \Biggl( \sum_{|i - n| > k} \bar{f}^i, \bar{f}^{n,k}   \Biggl)  \bar{f}^{n,m}   \right] \right|   \\
&\le \left| \mu \left[  ( \bar{f}^{n})^T   G^{\ve} \Biggl( \sum_{|i - n| > k} \bar{f}^i, \bar{f}^{n,k}   \Biggl)  \bar{f}^{n,m}   \right] \right| +  8M^2 \Vert G^{(\ve)} - G \Vert_\infty \\
&\le C_2 \left(   \Vert G^{\ve} \Vert_{\infty} + \Vert \nabla G^{\ve} \Vert_{\infty}  \right)\rho(m) \\
&+  8M^2 d K \lambda_{\min}^{-\frac32} (1 + \log N ) 6 \ve \log \ve^{-1} \\
&\le C_2 dK  \lambda_{\min}^{-\frac32}  \Biggl[ \left(    (1+  \log N  ) ( 4M + 1 + \ve )^2 + 1) +  \log (\ve^{-1}) 576 \cdot 4^d \right) \rho(m) \\
&+  8M^2  (1 + \log N ) 6 \ve \log \ve^{-1}  \biggr] \\
&\le C_2 d K \lambda_{\min}^{-\frac32} (1 + \log N) ( 97M^2  + 576 \cdot 4^d )  \log(\ve^{-1}) ( \rho(m) + \ve).
\end{align*}
Choosing $\ve =\tfrac12 \tfrac{\rho(m)}{\rho(0)} < 1$ implies condition (C2'). The proof for condition (C3') is omitted as it is almost verbatim the same.






\newpage

\bigskip
\bigskip
\bibliography{Sunklodas}{}
\bibliographystyle{plainurl}


\vspace*{\fill}

\end{document}